\documentclass[a4paper, titlepage, oneside, 11pt]{amsart}
\usepackage{amsmath,amssymb,amsfonts,amstext,amscd,cite,geometry,graphicx,latexsym,mathptmx,xcolor,pstricks}
\usepackage[utf8]{inputenc}
\usepackage{newunicodechar}
\usepackage[all]{xy}

\newtheorem{lem}{Lemma}[section]
\newtheorem{prop}[lem]{Proposition}
\newtheorem{thm}[lem]{Theorem}
\newtheorem{cor}[lem]{Corollary}
\newtheorem{conj}[lem]{Conjecture}
\newtheorem{rem}[lem]{Remark}
\newtheorem{exmp}[lem]{Example}
\newtheorem{assu}{Assumption}

\numberwithin{equation}{section}
\allowdisplaybreaks

\newcommand{\BA}{\circle{4}}
\newcommand{\BB}{\circle*{4}}

\newcommand{\BBBA}{\circle{15}}
\newcommand{\BBB}{\circle*{15}}

\begin{document}

\title[Invariant measures for the box-ball system]{Invariant measures for the box-ball system based on stationary Markov chains and periodic Gibbs measures}

\author[D.~A.~Croydon]{David A. Croydon}
\address{Department of Advanced Mathematical Sciences, Graduate School of Informatics, Kyoto University, Sakyo-ku, Kyoto 606--8501, Japan}
\email{croydon@acs.i.kyoto-u.ac.jp}

\author[M.~Sasada]{Makiko Sasada}
\address{Graduate School of Mathematical Sciences, University of Tokyo, 3-8-1, Komaba, Meguro-ku, Tokyo, 153--8914, Japan}
\email{sasada@ms.u-tokyo.ac.jp}

\begin{abstract}
The box-ball system (BBS) is a simple model of soliton interaction introduced by Takahashi and Satsuma in the 1990s. Recent work of the authors, together with Tsuyoshi Kato and Satoshi Tsujimoto, derived various families of invariant measures for the BBS based on two-sided stationary Markov chains \cite{CKSS}. In this article, we survey the invariant measures that were presented in \cite{CKSS}, and also introduce a family of new ones for periodic configurations that are expressed in terms of Gibbs measures. Moreover, we show that the former examples can be obtained as infinite volume limits of the latter. Another aspect of \cite{CKSS} was to describe scaling limits for the box-ball system; here, we review the results of \cite{CKSS}, and also present scaling limits other than those that were covered there. One, the zigzag process has previously been observed in the context of queuing; another, a periodic version of the zigzag process, is apparently novel. Furthermore, we demonstrate that certain Palm measures associated with the stationary and periodic versions of the zigzag process yield natural invariant measures for the dynamics of corresponding versions of the ultra-discrete Toda lattice.
\end{abstract}



\date{\today}

\maketitle

\section{Introduction}

The box-ball system (BBS) is an interacting particle system introduced in the 1990s by physicists Takahashi and Satsuma as a model to understand solitons, that is, travelling waves \cite{takahashi1990}. In particular, it is connected with the Korteweg-de Vries (KdV) equation, which describes shallow water waves; see \cite{TTMS} for background. The BBS is briefly described as follows. Initially, each site of the integer lattice $\mathbb{Z}$ contains a ball (or particle -- we will use the two terms interchangeably) or is vacant. For simplicity at this point, suppose there are only a finite number of particles. The system then evolves by means of a `carrier', which moves along the integers from left to right (negative to positive). When the carrier sees a ball it picks it up, and when it sees a vacant site it puts a ball down (unless it is not carrying any already, in which case it does nothing). See Figure \ref{bbsfig} for an example realisation.

To date, much of the interest in the BBS has come from applied mathematicians/theoretical physicists, who have established many beautiful combinatorial properties of the BBS, see \cite{IKT, T, TT} for introduction to such work. What has only recently started to be explored, however, are the probabilistic properties of the BBS resulting from a random initial starting configuration, see \cite{CKSS, Ferrari, FG, Lev} for essentially the only current literature on this topic. One particularly natural question in this direction is that of invariance, namely, which random configurations have a distribution that is invariant under the action of the box-ball system? In this article, we describe the invariant measures based on two-sided stationary Markov chains that were identified in \cite{CKSS}, and also introduce a family of new ones for periodic configurations that are expressed in terms of Gibbs measures.

\begin{figure}[b]
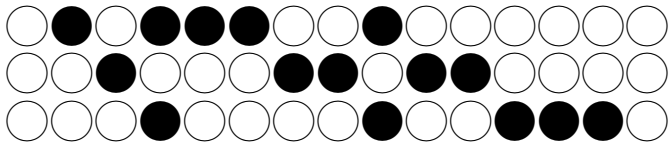

\begin{center}
{\BBBA$\:$\BBB$\:$\BBBA$\:$\BBB$\:$\BBB$\:$\BBB$\:$\BBBA$\:$\BBBA$\:$\BBB$\:$\BBBA$\:$\BBBA$\:$\BBBA$\:$\BBBA$\:$\BBBA$\:$\BBBA\\
\vspace{5pt}
\BBBA$\:$\BBBA$\:$\BBB$\:$\BBBA$\:$\BBBA$\:$\BBBA$\:$\BBB$\:$\BBB$\:$\BBBA$\:$\BBB$\:$\BBB$\:$\BBBA$\:$\BBBA$\:$\BBBA$\:$\BBBA\\
\vspace{5pt}
\BBBA$\:$\BBBA$\:$\BBBA$\:$\BBB$\:$\BBBA$\:$\BBBA$\:$\BBBA$\:$\BBBA$\:$\BBB$\:$\BBBA$\:$\BBBA$\:$\BBB$\:$\BBB$\:$\BBB$\:$\BBBA}
\end{center}
\caption{Two evolutions of the BBS. Black circles represent particles, white circles represent vacant sites.}\label{bbsfig}
\end{figure}

Given the transience of the system, i.e.\ each particle moves at least one position to the right on each time step of the dynamics, the question of invariance in distribution immediately necessitates the consideration of configurations $\eta=(\eta_n)_{n\in\mathbb{Z}}\in\{0,1\}^{\mathbb{Z}}$, where we write $\eta_n=1$ if there is a particle at location $n$ and $\eta_n=0$ otherwise, that incorporate an infinite number of particles on both the negative and positive axes. Of course, for such configurations, the basic description of the BBS presented above is no longer applicable, as one has to consider what it means for the carrier to traverse the integers from $-\infty$. This issue was addressed systematically in \cite{CKSS}, and at the heart of this study was a link between the BBS dynamics and the transformation of reflection in the past maximum that Pitman famously used to connect a one-dimensional Brownian motion with a three-dimensional Bessel process in \cite{Pitman}. We now describe this connection. Given a configuration $\eta\in\{0,1\}^{\mathbb{Z}}$, introduce a path encoding $S:\mathbb{Z}\rightarrow\mathbb{Z}$ by setting $S_0:=0$, and
\[S_n-S_{n-1}:=1-2\eta_n,\qquad \forall n\in\mathbb{Z},\]
and then define $TS:\mathbb{Z}\rightarrow\mathbb{Z}$ via the relation
\[(TS)_n:=2M_n-S_n-2M_0,\qquad \forall n\in\mathbb{Z},\]
where $M_n:=\sup_{m\leq n}S_m$ is the past maximum of $S$. Clearly, for the above formula to be well-defined, we require $M_0<\infty$. If this is the case, then we let $T\eta\in\{0,1\}^\mathbb{Z}$ be the configuration given by
\begin{equation}\label{teta}
(T\eta)_n:=\mathbf{1}_{\{(TS)_n-(TS)_{n-1}=-1\}},\qquad \forall n\in\mathbb{Z},
\end{equation}
(so that $TS$ is the path encoding of $T\eta$). It is possible to check that the map $\eta\mapsto T\eta$ coincides with the original definition of the BBS dynamics in the finite particle case \cite[Lemma 2.3]{CKSS}, and moreover is consistent with an extension to the case of a bi-infinite particle configuration satisfying $M_0<\infty$ from a natural limiting procedure \cite[Lemma 2.4]{CKSS}. We thus restrict to configurations for which $M_0<\infty$, and take \eqref{teta} as the definition of the BBS dynamics in this article. We moreover note that the process $W=(W_n)_{n\in\mathbb{Z}}$ given by
\[W_n:=M_n-S_n\]
can be viewed as the carrier process, with $W_n$ representing the number of balls transported by the carrier from $\{\dots,n-1,n\}$ to $\{n+1,n+2,\dots\}$; see \cite[Section 2.5]{CKSS} for discussion concerning the (non-)uniqueness of the carrier process.

Beyond understanding the initial step of the BBS dynamics, in the study of invariant random configurations it is natural to look for measures supported on the set of configurations for which the dynamics are well-defined for all times. Again, such an issue was treated carefully in \cite{CKSS}, with a full characterisation being given of the sets of configurations for which the one-step (forwards and backwards) dynamics are reversible (i.e.\ invertible), and for which the dynamics can be iterated for all time. Precisely, in \cite[Theorem 1.1]{CKSS} explicit descriptions were given for the sets:
\[\mathcal{S}^{rev}:=\left\{S\in \mathcal{S}^0\::\:\mbox{$TS$, $T^{-1}S$, $T^{-1}TS$, $TT^{-1}S$ well-defined, $T^{-1}TS=S$, $TT^{-1}S=S$}\right\},\]
where we have written $\mathcal{S}^0:=\{S:\mathbb{Z}\rightarrow \mathbb{Z}:\: S_0=0,\:|S_n-S_{n-1}|=1,\:\forall n\in\mathbb{Z}\}$ for the set of two-sided nearest-neighbour paths started from 0 (i.e.\ path encodings for configurations in $\{0,1\}^\mathbb{Z}$), and $T^{-1}$ for the inverse operation to $T$ that is given by `reflection in future minimum', see \cite[Section 2.6]{CKSS} for details; and also the invariant set
\[\mathcal{S}^{inv}:=\left\{S\in \mathcal{S}^{0}\::\:\mbox{$T^kS\in \mathcal{S}^{rev}$ for all $k\in\mathbb{Z}$}\right\}.\]
Whilst in this article we do not need to make full use of the treatment of these sets from \cite{CKSS}, we note the following important subset of path encodings
\begin{equation}\label{slin}
\mathcal{S}^{lin}:=\left\{S\in \mathcal{S}^{0}\::\:\lim_{|n|\rightarrow\infty}\frac{S_n}{n}=c\mbox{ for some }c>0\right\},
\end{equation}
consisting of asymptotically linear functions with a strictly positive drift. It is straightforward to check from the description given in \cite[Theorem 1.1]{CKSS} that $\mathcal{S}^{lin}\subseteq\mathcal{S}^{inv}\subseteq\mathcal{S}^{rev}$.

With the preceding preparations in place, we are ready to discuss directly the topic of invariance in distribution for random configurations, or equivalently particle encodings. In \cite{CKSS}, two approaches were pursued. The first was to relate the invariance of the BBS dynamics to the stationarity of the particle current across the origin, see \cite[Theorem 1.6]{CKSS}. Whilst the latter viewpoint does also provide an insight into the ergodicity of the transformation $\eta\mapsto T\eta$, in checking invariance in examples a more useful result was \cite[Theorem 1.7]{CKSS}, which relates the distributional invariance of $\eta$ under $T$ to two natural symmetry conditions -- one concerning the configuration itself, and one concerning the carrier process. In particular, to state the result in question, we introduce the reversed configuration $\overleftarrow{\eta}$, as defined by setting
\[\overleftarrow{\eta}_n=\eta_{-(n-1)},\]
and the reversed carrier process $\bar{W}$, given by
\[\bar{W}_n=W_{-n}.\]

\begin{thm}[See {\cite[Theorem 1.7]{CKSS}}]\label{mrd} Suppose $\eta$ is a random particle configuration, and that the distribution of the corresponding path encoding $S$ is supported on $\mathcal{S}^{rev}$. It is then the case that any two of the three following conditions imply the third:
\begin{equation}\label{threeconds}
\overleftarrow{\eta}\buildrel{d}\over{=}\eta,\qquad \bar{W}\buildrel{d}\over{=}W,\qquad T\eta\buildrel{d}\over{=}\eta.
\end{equation}
Moreover, in the case that two of the above conditions are satisfied, then the distribution of $S$ is actually supported on $\mathcal{S}^{inv}$.
\end{thm}

As an application of the previous result, the following  fundamental examples of invariant random configurations were presented in \cite[Theorem 1.8]{CKSS}:
\begin{itemize}
  \item The particle configuration $(\eta_n)_{n\in\mathbb{Z}}$ given by a sequence of independent identically distributed (i.i.d.) Bernoulli random variables with parameter $p\in[0,\frac{1}{2})$.
  \item The particle configuration $(\eta_n)_{n\in\mathbb{Z}}$ given by a two-sided stationary Markov chain on $\{0,1\}$ with transition matrix
\[\left(
  \begin{array}{cc}
    1-p_0 & p_0 \\
    1-p_1 & p_1 \\
  \end{array}
\right)\]
where $p_0\in (0,1)$, $p_1\in[0,1)$ satisfy $p_0+p_1<1$.
  \item For any $K\in\mathbb{Z}_+$, the particle configuration $(\eta_n)_{n\in\mathbb{Z}}$ given by conditioning a sequence of i.i.d.\ Bernoulli random variables with parameter $p\in(0,1)$ on the event $\sup_{n\in\mathbb{Z}}W_n\leq K$.
\end{itemize}
Further details of these are recalled in Subsections \ref{iidsec}-\ref{boundedsec}, respectively. Another easy example, discussed in \cite[Remark 1.13]{CKSS}, arises from a consideration of the periodic BBS introduced in \cite{YT} -- that is, the BBS that evolves on the torus $\mathbb{Z}/N\mathbb{Z}$. As commented in \cite{CKSS}, if we repeat a configuration of length $N$ with strictly fewer than $N/2$ balls in a cyclic fashion, then we obtain a configuration with path encoding contained in $\mathcal{S}^{lin}$, and, by placing equal probability on each of the distinct configurations that we see as the BBS evolves, we obtain an invariant measure for the system.

Now, it should be noted that \cite{CKSS} was not the first study to identify the first two configurations above (i.e.\ the i.i.d.\ and Markov configurations) as invariant under $T$. Such results had previously been established in queueing theory -- the invariance of the i.i.d.\ configuration can be seen as a discrete time analogue of the classical theorem of Burke \cite{Burke}, and the invariance of the Markov configuration was essentially proved in \cite{HMOC}. However, in the study of invariants for Pitman's transformation, the BBS does add an important new perspective -- the central role of solitons. Indeed, in the original study of \cite{takahashi1990}, it was observed that configurations can be decomposed into a collection of ‘basic strings’ of the form (1,0), (1,1,0,0), (1,1,1,0,0,0), etc., which act like solitons in that they are preserved by the action of the carrier, and travel at a constant speed (depending on their length) when in isolation, but experience interactions when they meet. Moreover, in the enlightening recent work of \cite{Ferrari} (where the invariance of the i.i.d.\ configuration was again observed), it was conjectured that any invariant measure on configurations can be decomposed into independent measures on solitons of different sizes. (The latter study investigated the speeds of solitons in invariant random configurations under continued evolution of the BBS system.) See also \cite{FG} for a related follow-up work.

Motivated in part by \cite{Ferrari}, in this article we introduce a class of invariant periodic configurations whose laws are described in terms of Gibbs measures involving a soliton decomposition. (These were already described formally in \cite[Remark 1.12]{CKSS}, and are closely paralleled by the measures studied in \cite{FG}.) Specifically, we first fix a cycle length $N\in\mathbb{N}$, and then define a random variable $(\eta^N_n)_{n=1}^{N}$ taking values in $\{0,1\}^N$ by setting
\begin{equation}\label{gibbs}
\mathbf{P}\left((\eta^N_n)_{n=1}^{N}=(x_n)_{n=1}^{N}\right)=\frac{1}{Z}\exp\left(-\sum_{k=0}^\infty\beta_kf_k\left((x_n)_{n=1}^{N}\right)\right)\mathbf{1}_{\{f_0\left((x_n)_{n=1}^{N}\right)<N/2\}},
\end{equation}
for $(x_n)_{n=1}^{N}\in\{0,1\}^N$, where $\beta_k\in\mathbb{R}\cup\{\infty\}$ for each $k\geq 0$,
\[f_0\left((x_n)_{n=1}^{N}\right):=\#\left\{\mbox{particles in $(x_n)_{n=1}^{N}$}\right\},\]
\[f_k\left((x_n)_{n=1}^{N}\right):=\#\left\{\mbox{solitons of size $\geq k$ in $(x_n)_{n=1}^{N}$}\right\},\qquad \forall k\in\mathbb{N},\]
and $Z$ is a normalising constant. We then extend to $\eta^N=(\eta^N_n)_{n\in\mathbb{Z}}$ by cyclic repetition; the law of $\eta^N$ is our Gibbs measure. (Further details are provided in Subsection \ref{gibbssec}.) The invariance under $T$ of such a random configuration is checked as Corollary \ref{gibbsinv} below. Moreover, in Proposition \ref{infvollim}, it is shown that each of the three configurations of \cite[Theorem 1.8]{CKSS} can be obtained as an infinite volume ($N\rightarrow\infty$) limit of these periodic configurations.

\begin{rem} In this article, we are using the term `Gibbs measure' in a loose sense. Given that the expression at \eqref{gibbs} incorporates the infinite number of conserved quantities for the integrable system that is the BBS, following \cite{RDYO, RMO} (see also the review \cite{VR}), it might rather be seen as a `generalised Gibbs measure'. Since we plan to present a more comprehensive study of Gibbs-type measures for the BBS in a following article, we leave further discussion of this point until the future.
\end{rem}

The description of the path encoding of a configuration and its evolution under the BBS dynamics provides a convenient framework for deriving scaling limits. In \cite{CKSS}, the most natural example from the point of view of probability theory, in which the path encodings of a sequence of i.i.d.\ configurations of increasing density were rescaled to a two-sided Brownian motion with drift, was presented. Not only did the latter result provide a means to establishing the invariance of two-sided Brownian motion with drift under Pitman's transformation (a result which was already known from the queuing literature,  see \cite[Theorem 3]{OCY}, and \cite{HW} for an even earlier proof), but it provided motivation to introduce a model of BBS on $\mathbb{R}$. (A particular version of the latter model is checked to be integrable in \cite{CST}.) Specifically, this was given by applying Pitman's transformation to elements $S$ of $C(\mathbb{R},\mathbb{R})$ satisfying $S_0=0$ and $\sup_{t\leq 0}S_t<\infty$. In this article, we recall the aforementioned scaling limit (see Subsection \ref{BMsec}), and also give its periodic variant (see Subsection \ref{perbmsec}), as well as discuss a continuous version of the bounded soliton example (see Subsection \ref{contbounded}). As another important example, we describe a parameter regime in which the Markov configuration can be rescaled to the zigzag process, which consists of straight line segments of random length and alternating gradient $+1$ or $-1$  (see Subsection \ref{zzsec}). The description of the latter process as a scaling limit readily yields its invariance under Pitman's transformation (this result also appears in the queueing literature, see \cite{HMOC}). We also give a periodic version of zigzag process, show it is a scaling limit of cyclic Markov configurations, and establish its invariance under Pitman's transformation -- a result that we believe is new (see Subsection \ref{perzzsec}). From the point of view of integrable systems, the transformation of the zigzag process (and its periodic counterpart) under Pitman's transformation can be seen as describing the dynamics of the ultra-discrete Toda lattice (and its periodic counterpart, respectively) started from certain random initial conditions. By considering certain Palm measures associated with the zigzag process, the results of this article give natural invariant probability measures for the latter system as well (see Section \ref{palmsec}).

The remainder of this article is organised as follows. In Section \ref{discsec}, we present our examples of discrete invariant measures for the transformation $\eta\mapsto T\eta$. In Section \ref{contsec}, we detail the scaling limit framework, and explain how this can be applied to deduce invariance under Pitman's transformation of various random continuous stochastic processes. In Section \ref{palmsec}, we introduce Palm measures for the zigzag process, and use these to derive invariant measures for the ultra-discrete Toda lattice. Finally, in Section \ref{condsec}, we give a brief presentation concerning the connection between invariance under $T$ for a two-sided process and the laws of a conditioned versions of the corresponding one-sided process. NB. Regarding notational conventions, we write $\mathbb{N}=\{1,2,3,\dots\}$ and $\mathbb{Z}_+=\{0,1,2,\dots\}$.

\section{Discrete invariant measures}\label{discsec}

In the first part of this section (Subsections \ref{iidsec}-\ref{boundedsec}), we recall the invariant measures for the box-ball system (or equivalently the discrete-space version of Pitman's transformation) that were studied in \cite{CKSS}. As established in \cite{CKSS}, these represent all the invariant measures whose path encodings are supported on $\mathcal{S}^{rev}$ for which either the configuration $\eta$ or the carrier process $W$ is a two-sided stationary Markov chain (see \cite[Remark 1.10]{CKSS} in particular). Following this, in Subsection \ref{gibbssec}, we introduce a family of new invariant measures on periodic configurations based on certain Gibbs measures, and show that all the earlier examples can be obtained as infinite volume limits of these.

\subsection{Independent and identically distributed initial configuration}\label{iidsec} Perhaps the most fundamental invariant measure for the box-ball system is the case when $\eta$ is given by a sequence of independent and identically distributed Bernoulli random variables with parameter $p$. To ensure the law of the associated path encoding has distribution supported on $\mathcal{S}^{lin}$ (as defined at \eqref{slin}), we require $p<\frac12$. It is also clear that $\overleftarrow{\eta}\buildrel{d}\over{=}\eta$, and so the first of the conditions at \eqref{threeconds} is fulfilled. Moreover, the second of the conditions at \eqref{threeconds}, i.e.\ that $\bar{W}\buildrel{d}\over{=}W$, readily follows from the following description of the carrier process $W$ as a Markov chain. Indeed, the equations \eqref{Wprobs} and \eqref{pidef} below imply that detailed balance is satisfied by $W$, and thus it is reversible. As a result, Theorem \ref{mrd} can immediately be applied to deduce the invariance of the i.i.d.\ configuration, which we state precisely as Corollary \ref{revcor}.

{\lem[See {\cite[Lemma 3.13]{CKSS}}] If $\eta$ is given by a sequence of i.i.d.\ Bernoulli($p$) random variables with $p\in[0,\frac12)$, then $W$ is a two-sided stationary Markov chain with transition probabilities given by
\begin{equation}\label{Wprobs}
\mathbf{P}\left(W_{n}=W_{n-1}+j\:\vline\: W_{n-1}\right)=\left\{\begin{array}{ll}
                                                p, & \mbox{if }j=1,\\
                                                1-p, & \mbox{if }W_{n-1}>0\mbox{ and }j=-1,\\
                                                1-p, & \mbox{if }W_{n-1}=0\mbox{ and }j=0.\\
                                              \end{array}\right.
\end{equation}
The stationary distribution of this chain is given by $\pi=(\pi_x)_{x\in\mathbb{Z}_+}$, where
\begin{equation}\label{pidef}
\pi_x=\left(\frac{1-2p}{1-p}\right)\left(\frac{p}{1-p}\right)^x,\qquad \forall x\in\mathbb{Z}_+.
\end{equation}
}

{\cor[See {\cite[Corollary 3.14]{CKSS}}]\label{revcor} If $\eta$ is a sequence of i.i.d.\ Bernoulli($p$) random variables with $p\in[0,\frac12)$, then the three conditions of \eqref{threeconds} are satisfied. In particular, $\eta$ is invariant in distribution under $T$.}

\subsection{Markov initial configuration}\label{markovsec} As a generalisation of the i.i.d.\ configuration of the previous section, we next consider the case when $\eta$ is a two-sided stationary Markov chain on $\{0,1\}$ with transition matrix
\begin{equation}\label{pdef}
P=\left(
      \begin{array}{cc}
        1-p_0 & p_0 \\
        1-p_1 & p_1 \\
      \end{array}
    \right),
\end{equation}
by which we mean
\[\mathbf{P}\left(\eta_{n+1}=1\:\vline\:\eta_n=j\right)=p_j,\qquad j\in\{0,1\},\]
for some parameters $p_0\in (0,1)$, $p_1\in[0,1)$. Note that we recover the i.i.d.\ case when $p_0=p_1=p$. The stationary distribution of this chain is given by
\begin{equation}\label{rhodef}
\rho=\mathbf{P}\left(\eta_0=1\right)=\frac{p_0}{1-p_1+p_0},
\end{equation}
and so to ensure the associated path encoding has distribution supported on $\mathcal{S}^{lin}$, we thus need to assume $p_0+p_1<1$. Since detailed balance is satisfied by $\eta$, we have that $\overleftarrow{\eta}\buildrel{d}\over{=}\eta$. Moreover, although $W$ is not a Markov chain, it is a stationary process whose marginal distributions are given by the following lemma, and \cite[Theorem 2]{HMOC} gives that $\bar{W}\buildrel{d}\over{=}W$. Thus we obtain from another application of Theorem \ref{mrd} the generalisation of Corollary \ref{revcor} to the Markov case, see Corollary \ref{markovcor} below.

{\lem[See {\cite[Lemma 3.15]{CKSS}}]\label{w0pm} If $\eta$ is the two-sided stationary Markov chain described above with
$p_0\in (0,1)$, $p_1\in[0,1)$ satisfying $p_0+p_1<1$, then
\[\mathbf{P}\left(W_{0}=m\right)=\left\{\begin{array}{ll}
                                                \frac{1-p_0-p_1}{(1-p_0)(1+p_0-p_1)}, & \mbox{if }m=0,\\
                                                \frac{p_0(1-p_0+p_1)(1-p_0-p_1)}{(1-p_0)^2(1+p_0-p_1)}\left(\frac{p_1}{1-p_0}\right)^{m-1}, & \mbox{if }m\geq 1.\\
                                              \end{array}\right.\]}

{\cor[See {\cite[Corollary 3.16]{CKSS}}]\label{markovcor} If $\eta$ is the two-sided stationary Markov chain described above with
$p_0\in (0,1)$, $p_1\in[0,1)$ satisfying $p_0+p_1<1$, then the three conditions of (\ref{threeconds}) are satisfied. In particular, $\eta$ is invariant in distribution under $T$.}

\subsection{Conditioning the i.i.d.\ configuration to have bounded solitons}\label{boundedsec} In the two previous examples, it is possible to check that $\sup_{n\in\mathbb{Z}}W_n=\infty$, $\mathbf{P}$-a.s., which can be interpreted as meaning that the configurations admit solitons of an unbounded size. The motivation for the introduction of the example we present in this section came from the desire to exhibit a random initial configuration that contained solitons of a bounded size. To do this, the approach of \cite{CKSS} was to condition the i.i.d.\ configuration of Section \ref{iidsec} to not contain any solitons of size greater than $K$, or equivalently that $\sup_{n\in\mathbb{Z}}W_n\leq K$, for some fixed $K\in\mathbb{Z}_+$. Since the latter is an event of 0 probability whenever $\eta$ is Bernoulli($p$), for any $p\in (0,1)$, a limiting argument was used to make the this description rigourous. In particular, applying the classical theory of quasi-stationary distributions for Markov chains, we were able to show that the resulting configuration $\tilde{\eta}^{(K)}$ is stationary, ergodic, has path encoding with distribution supported on $\mathcal{S}^{lin}$, and moreover the three conditions at (\ref{threeconds}) hold.

To describe the construction of $\tilde{\eta}^{(K)}$ precisely, we start by defining the associated carrier process. Let $P=(P(x,y))_{x,y\in\mathbb{Z}_+}$ be the transition matrix of $W$, as defined in (\ref{Wprobs}) (where we now allow any $p\in (0,1)$).  For $K\in \mathbb{Z}_+$ fixed, let $P^{(K)}=(P^{(K)}(x,y))_{x,y\in\{0,\dots,K\}}$ be the restriction of $P$ to $\{0,\dots,K\}$. Since $P^{(K)}$ is a finite, irreducible, substochastic matrix, it admits (by the Perron-Frobenius theorem) a unique eigenvalue of largest magnitude, $\lambda_K$ say. Moreover, $\lambda_K\in(0,1)$ and has a unique (up to scaling) strictly positive eigenvector $h_K=(h_K(x))_{x\in\{0,\dots,K\}}$. Let $\tilde{P}^{(K)}=(\tilde{P}^{(K)}(x,y))_{x,y\in\{0,\dots,K\}}$ be the stochastic matrix defined by
\[\tilde{P}^{(K)}(x,y)=\frac{{P}^{(K)}(x,y)h_K(y)}{\lambda_Kh_K(x)},\qquad \forall x,y\in \{0,\dots,K\}.\]
The associated Markov chain is reversible, and has stationary probability measure given by $\tilde{\pi}^{(K)}=(\tilde{\pi}^{(K)}_x)_{x\in \{0,\dots,K\}}$, where $\tilde{\pi}^{(K)}_x = c_1 h_K(x)^2{\pi}_x$ for some constant $c_1\in(0,\infty)$ (which may depend on $K$), and $\pi$ is defined as at (\ref{pidef}). Thus the Markov chain in question admits a two-sided stationary version, and we denote this by $\tilde{W}^{(K)}=(\tilde{W}^{(K)}_n)_{n\in\mathbb{Z}}$. We view $\tilde{W}^{(K)}$ as a random carrier process, and write the associated particle configuration $\tilde{\eta}^{(K)}=(\tilde{\eta}^{(K)}_n)_{n\in\mathbb{Z}}$.

To justify the claim that $\tilde{\eta}^{(K)}$ is the i.i.d.\ configuration of Section \ref{iidsec} conditioned to have solitons of size no greater than $K$, we have the following result. (An alternative description of the limit that is valid for $p\in(0,\frac12)$ is given in \cite[Remark 3.18]{CKSS}.)

{\lem[See {\cite[Lemma 3.17]{CKSS}}] Fix $K\in \mathbb{Z}_+$. Let ${\eta}=({\eta}_n)_{n\in\mathbb{Z}}$ be an i.i.d.\ Bernoulli($p$) particle configuration for some $p\in (0,1)$. Write $\eta^{[-N,N]}=(\eta^{[-N,N]}_n)_{n\in\mathbb{Z}}$ for the truncated configuration given by $\eta^{[-N,N]}_n=\eta_n\mathbf{1}_{\{-N<n\leq N\}}$. If $W^{[-N,N]}$ is the associated carrier process, then we have the following convergence of conditioned processes:
\[W^{[-N,N]}\:\vline\: \left\{\sup_{n\in\mathbb{Z}}W^{[-N,N]}_n\leq K\right\}\rightarrow\tilde{W}^{(K)}\]
in distribution as $N\rightarrow \infty$. In particular, this implies
\[\eta^{[-N,N]}\:\vline\: \left\{\sup_{n\in\mathbb{Z}}W^{[-N,N]}_n\leq K\right\}\rightarrow\tilde{\eta}^{(K)}\]
in distribution as $N\rightarrow \infty$.}
\bigskip

As a consequence of the construction of $\tilde{\eta}^{(K)}$, it is possible to check the following result.

{\cor[See {\cite[Corollary 3.19]{CKSS}}]\label{boundedcor} If $\tilde{\eta}^{(K)}$ and $\tilde{W}^{(K)}$ are as described above, then, for any $p\in (0,1)$, $K\in\mathbb{Z}_+$, $\tilde{\eta}^{(K)}$ is a stationary, ergodic process satisfying
\[\mathbf{P}\left(\tilde{\eta}^{(K)}_0=1\right)<\frac12,\]
and also the three conditions of (\ref{threeconds}). In particular, $\tilde{\eta}^{(K)}$ is invariant in distribution under $T$.}

\subsection{Initial configurations given by periodic Gibbs measures}\label{gibbssec} To define the Gibbs measures of interest, we start by introducing functions to count the number of solitons of certain sizes within the cycle of a periodic configuration. In particular, we first fix $N\in\mathbb{N}$ to represent our cycle length, and define
\[\begin{array}{rcl}
f_0 : \{0,1\}^N& \to& \mathbb{Z}_+\\
(x_n)_{n=1}^{N}&\mapsto& \sum_{n=1}^{N}x_n,
\end{array}\]
which will count the number of particles within a cycle of a periodic configuration. Next, we introduce
\[\begin{array}{rcl}
f_1 : \{0,1\}^N& \to& \mathbb{Z}_+\\
(x_n)_{n=1}^{N}&\mapsto& \sum_{n=1}^{N}\mathbf{1}_{\{x_{n-1}=1,x_{n}=0\}},
\end{array}\]
where we suppose $x_0:=x_N$ for the purposes of the above formula; this function will count the number of solitons within a cycle of a periodic configuration. To define $f_k$ for higher values of $k$, we introduce a contraction operation on particle configurations. Specifically, given a finite length configuration $(x_n)_{n=1}^{m}$ of $0$s and $1$s, define a new configuration $H((x_n)_{n=1}^{m})$ by removing all $(1,0)$ strings from $(x_n)_{n=1}^{m}$, including the pair $(x_{m},x_1)$ if relevant. For $k\geq 2$, we then set
\[\begin{array}{rcl}
f_k : \{0,1\}^N& \to& \mathbb{Z}_+\\
(x_n)_{n=1}^{N}&\mapsto& f_1\left(H^{k-1}\left((x_n)_{n=1}^{N}\right)\right),
\end{array}\]
where the definition of $f_1$ is extended to finite strings of arbitrary length in the obvious way; this function will count the number of solitons of length at least $k$ within a cycle of a periodic configuration. That $f_k$ describe conserved quantities for the box-ball system and indeed have the desired soliton interpretation, see \cite{YYT} (cf.\ the corresponding description in the non-periodic case of \cite{Torii}, and the description of the number of solitons of certain lengths via the `hill-flattening' operator of \cite{Lev}). We subsequently define a random variable $(\eta^N_n)_{n=1}^{N}$ taking values in $\{0,1\}^N$ by setting, as initially presented at \eqref{gibbs},
\[\mathbf{P}\left((\eta^N_n)_{n=1}^{N}=(x_n)_{n=1}^{N}\right)=\frac{1}{Z}\exp\left(-\sum_{k=0}^\infty\beta_kf_k\left((x_n)_{n=1}^{N}\right)\right)\mathbf{1}_{\{f_0\left((x_n)_{n=1}^{N}\right)<N/2\}}\]
for $(x_n)_{n=1}^{N}\in\{0,1\}^N$, where $\beta_k\in\mathbb{R}\cup\{\infty\}$ for each $k\geq 0$ and $Z$ is a normalising constant. NB.\ To ensure the measure is well-defined, we adopt the convention that if $\beta_k=\infty$ and $f_k((x_n)_{n=1}^{N})=0$, then their product is zero. We then extend to $\eta^N=(\eta^N_n)_{n\in\mathbb{Z}}$ by cyclic repetition; the law of $\eta^N$ is our Gibbs measure. Clearly, the inclusion of the term $\mathbf{1}_{\{f_0((x_n)_{n=1}^{N})<N/2\}}$ yields that the distribution of the path encoding of the configuration $\eta^N$ is supported on $\mathcal{S}^{lin}$.

We next check the spatial stationarity and distributional symmetry of $\eta^N$, and the distributional symmetry of the associated carrier process $W^N$.

\begin{lem}\label{etalem} The law of the periodic configuration $\eta^N$, as described by the Gibbs measure at \eqref{gibbs}, is stationary under spatial shifts. Moreover, $\overleftarrow{\eta}^N\buildrel{d}\over{=}\eta^N$.
\end{lem}
\begin{proof} For $x=(x_n)_{n=1}^{N}\in\{0,1\}^N$, it is straightforward to check from the definitions of the relevant functions that
\begin{equation}\label{per}
f_k(x)=f_k(\theta_{Per}x),\qquad \forall k\geq 0,
\end{equation}
where $\theta_{Per}$ is the periodic shift operator given by $\theta_{Per}x:=(x_2,\dots,x_N,x_1)$. Hence we obtain from \eqref{gibbs} that
\[\mathbf{P}\left(\theta_{Per}\left((\eta^N_n)_{n=1}^{N}\right)=x\right)=\mathbf{P}\left((\eta^N_n)_{n=1}^{N}=x\right),\qquad \forall x\in\{0,1\}^N.\]
It readily follows that $\theta\eta^N\buildrel{d}\over{=}\eta^N$, where $\theta$ is the left-shift on doubly infinite sequences, i.e.\ $\theta((x_n)_{n\in\mathbb{Z}})=(x_{n+1})_{n\in\mathbb{Z}}$. This establishes the first claim of the lemma.

We now check the second claim. For $x=(x_n)_{n=1}^{N}\in\{0,1\}^N$, write $\overleftarrow{x}$ for the reversed sequence $(x_{N+1-n})_{n=1}^N$. We clearly have that
\[f_0\left(x\right)=f_0\left(\overleftarrow{x}\right).\]
Moreover, recall that $f_1$ counts the number of $(1,0)$ strings in $x$, including the $(x_N,x_1)$ pair. The latter periodicity readily implies that this is equal to the number of $(0,1)$ strings in $x$ (cf.\ \cite[Lemma 2.1]{YYT}). Hence
\begin{equation}\label{f1prop}
f_1\left(x\right)=f_1\left(\overleftarrow{x}\right).
\end{equation}
Next, further recall that the configuration $H(x)$ is obtained by removing all $(1,0)$ strings from $x$, including the pair $(x_{N},x_1)$ if relevant. Since this operation simply reduces the lengths of all the strings of consecutive $0$ strings of consecutive $1$s by one, it is the same (up to a periodic shift) as the $(0,1)$-removal operation; this observation was made in \cite{YYT} (below Lemma 2.1 of that article), and also in the proof of \cite[Lemma 2.1]{Lev} in the non-periodic case. In particular, we have that
\[H(x)=\theta_{Per}^{l_x}\overleftarrow{H\left(\overleftarrow{x}\right)}\]
for some integer $l_x$ (where the definition of the periodic shift operator is extended to finite sequences of arbitrary length in the obvious way). Hence, applying this observation in conjunction with \eqref{per} and \eqref{f1prop}, we find that
\[f_k\left(\overleftarrow{x}\right)=f_1\left(H^{k-1}\left(\overleftarrow{x}\right)\right)=
f_1\left(\overleftarrow{H^{k-1}\left({x}\right)}\right)=f_1\left({H^{k-1}\left({x}\right)}\right)=f_k(x).\]
As a consequence of these observations, we thus obtain
\[\mathbf{P}\left((\eta^N_n)_{n=1}^{N}=x\right)=\mathbf{P}\left((\eta^N_n)_{n=1}^{N}=\overleftarrow{x}\right),\qquad \forall x\in\{0,1\}^N,\]
which implies $\overleftarrow{\eta}^N\buildrel{d}\over{=}\eta^N$, as desired.
\end{proof}

\begin{lem}\label{wlem} If $\eta^N$ is the periodic configuration with law given by the Gibbs measure at \eqref{gibbs}, then $\bar{W}^N\buildrel{d}\over{=}W^N$.
\end{lem}
\begin{proof} For a sequence $w:\{1,\dots,N\}\rightarrow \mathbb{Z}_+$, define the associated periodic increment process $\Delta(w)=(\Delta(w)_n)_{n=1}^N$ by setting
\[\Delta(w)_n=w_n-w_{n-1},\qquad \forall n\in\{1,\dots,N\},\]
where we define $w_0:=w_N$. Moreover, let $\mathcal{W}$ be the set of $w:\{1,\dots,N\}\rightarrow \mathbb{Z}_+$ such that $\Delta(w)\in\{-1,0,1\}^N$, $\Delta(w)_n=0$ if and only if $w_n=w_{n-1}=0$, and $\Delta(w)_n=0$ for at least one $n\in\{1,\dots,N\}$. Note that, on this set, $w$ is uniquely determined by $\Delta(w)$.

Now, since the configuration is $N$-periodic and $S_N>0$, $W^N$ is also $N$-periodic and moreover $(W_n)_{n=1}^N$ takes values in $\mathcal{W}$, $\mathbf{P}$-a.s. Since $\Delta(W^N)_n=1$ if and only if $\eta^N_n=1$, it follows that, for all $w\in\mathcal{W}$,
\[\mathbf{P}\left((W^N_n)_{n=1}^N=w\right)=\mathbf{P}\left((\Delta(W^N)_n)_{n=1}^N=\Delta(w)\right)=\mathbf{P}\left((\eta^N_n)_{n=1}^N=x\right),\]
where $x=(x_n)_{n=1}^N$ is defined by setting $x_n:=\mathbf{1}_{\{\Delta(w)_n=1\}}$. Moreover, using the notation $\bar{w}=(w_{N-1},w_{N-2},\dots,w_1,w_N)$ (which is also an element of $\mathcal{W}$), we have that
\[\mathbf{P}\left((\bar{W}^N_n)_{n=1}^N=w\right)=\mathbf{P}\left(({W}^N_n)_{n=1}^N=\bar{w}\right)=
\mathbf{P}\left((\Delta(W^N)_n)_{n=1}^N=\Delta(\bar{w})\right).\]
A simple calculation yields that $\Delta(\bar{w})=-\overleftarrow{\Delta(w)}$, and so we find that
\[\mathbf{P}\left((\bar{W}^N_n)_{n=1}^N=w\right)=\mathbf{P}\left((\eta^N_n)_{n=1}^N=\bar{x}\right),\]
where $\bar{x}=(\bar{x}_n)_{n=1}^N$ is defined by setting $\bar{x}_n:=\mathbf{1}_{\{\overleftarrow{\Delta(w)}_n=-1\}}$. In particular, the result will follow from the above observations and \eqref{gibbs} if we can show that $f_k(x)=f_k(\bar{x})$ for each $k\geq 0$.

Clearly, periodicity implies that the number of up-jumps of $w$ equals the number of down-jumps, and so
\[f_0(\bar{x})=\sum_{n=1}^N\mathbf{1}_{\left\{\overleftarrow{\Delta(w)}_n=-1\right\}}=
\sum_{n=1}^N\mathbf{1}_{\left\{{\Delta(w)}_n=-1\right\}}=\sum_{n=1}^N\mathbf{1}_{\left\{{\Delta(w)}_n=1\right\}}=f_0(x).\]
Furthermore, since $\Delta(w)$ can not contain the substrings $(0,-1)$ or $(1,0)$,
\begin{eqnarray*}
f_1(\bar{x})&=&\sum_{n=1}^N\mathbf{1}_{\left\{\overleftarrow{\Delta(w)}_{n-1}=-1,\:\overleftarrow{\Delta(w)}_n\in\{0,1\}\right\}}\\
&=&\sum_{n=1}^N\mathbf{1}_{\left\{{\Delta(w)}_{n-1}\in\{0,1\},\:{\Delta(w)}_n=-1\right\}}\\
&=&\sum_{n=1}^N\mathbf{1}_{\left\{{\Delta(w)}_{n-1}=1,\:{\Delta(w)}_n\in\{-1,0\}\right\}}\\
&=&f_1(x)
\end{eqnarray*}
Finally, observe that the $(1,-1)$ substrings of $\Delta(w)$ (including the one at $(w_N,w_1)$ if relevant) precisely correspond to the $(1,0)$ substrings of $x$. Moreover, if we suppose $H_W$ is the operation which removes these substrings, then it is an easy exercise to check that $H_W(\Delta(w))$ is the element of $\mathcal{W}$ representing the periodic increment process of the carrier associated with the configuration given by $H(x)$. We can iterate this argument to further obtain that $H_W^{k-1}(\Delta(w))$ is the element of $\mathcal{W}$ representing the periodic increment process of the carrier associated with the configuration given by $H^{k-1}(x)$ for any $k\geq 2$. Hence we can write
\begin{equation}\label{fk1}
f_k({x})=f_1\left(H^{k-1}({x})\right)=f_1\left(\left(\mathbf{1}_{\{H^{k-1}_W(\Delta({w}))_n=1\}}\right)_{n=1}^l\right),
\end{equation}
where $l$ is the length of the sequence $H_W^{k-1}(\Delta(w))$. Applying the same logic to $\bar{w}$, we similarly have that $H_W^{k-1}(\Delta(\bar{w}))$ is the element of $\mathcal{W}$ representing the periodic increment process of the carrier associated with the configuration given by $H^{k-1}(\bar{x})$ for any $k\geq 2$, and moreover the definition of $H_W$ readily implies that
\[H_W^{k-1}(\Delta(\bar{w}))=H_W^{k-1}\left(-\overleftarrow{\Delta({w})}\right)=-\overleftarrow{H_W^{k-1}(\Delta({w}))}.\]
Hence
\begin{equation}\label{fk2}
f_k(\bar{x})=f_1\left(\left(\mathbf{1}_{\{-\overleftarrow{H^{k-1}_W(\Delta({w}))}_n=1\}}\right)_{n=1}^l\right),
\end{equation}
and the argument for $f_1$ above shows the right-hand side of \eqref{fk1} and \eqref{fk2} are equal, which completes the proof.
\end{proof}

As a consequence of the previous two lemmas and Theorem \ref{mrd}, we readily obtain the main result of this section.

\begin{cor}\label{gibbsinv} If $\eta^N$ is the periodic configuration with law given by the Gibbs measure at \eqref{gibbs}, then the three conditions of (\ref{threeconds}) are satisfied. In particular, $\eta^N$ is invariant in distribution under $T$.
\end{cor}

\begin{rem} We now discuss an alternative, direct proof of Corollary \ref{gibbsinv}. Let $x\in\{0,1\}^N$ be such that $f_0(x)<N/2$, and $Tx=((Tx)_n)_{n=1}^N$ be the image of $x$ under the action of the periodic BBS. The definitions readily yield that if $w$ is the carrier path associated with $x$, then
\[Tx=\left(\mathbf{1}_{\Delta(w)_n=-1}\right)_{n=1}^N=\overleftarrow{\bar{x}},\]
where we are using the notation of the proofs of Lemmas \ref{etalem} and \ref{wlem}. Moreover, the arguments applied in these proofs imply that
\begin{equation}\label{fkpres}
f_k(Tx)=f_k(x),\qquad \forall k\geq 0.
\end{equation}
It clearly follows that the Gibbs measure at (\ref{gibbs}) is invariant under $T$, and we arrive at Corollary \ref{gibbsinv}. We note that the identity at (\ref{fkpres}) was previously proved as \cite[Proposition 2.1]{YYT}, see also \cite{Torii} for a proof in the non-periodic case.
\end{rem}

To conclude this section, we relate the Gibbs measures of this section with the i.i.d., Markov and bounded soliton configurations of Subsections \ref{iidsec}, \ref{markovsec} and \ref{boundedsec}, respectively. In particular, in the following examples we introduce three specific parameter choices for the Gibbs measures, and then show in Proposition \ref{infvollim} below that the aforementioned configurations can be obtained as infinite volume limits of these. Moreover, in Subsections \ref{perbmsec}, \ref{perzzsec} and \ref{contbounded}, we present scaling limits for certain sequences of periodic configurations based on these examples.

\begin{exmp}[Periodic i.i.d. initial configuration]\label{periidex} Similarly to \cite[Remark 1.12]{CKSS}, let $p\in (0,1)$, and consider the parameter choice
\[\beta_0=\log\left(\frac{1-p}{p}\right),\qquad \beta_k=0,\:\forall k \ge 1.\]
(Figure \ref{iidfig} shows a typical realisation of a configuration chosen according the associated Gibbs measure, and its subsequent evolution.) It is then an elementary exercise to check that
\begin{equation}\label{iidexp}
\mathbf{P}\left((\eta^N_n)_{n=1}^{N}=(x_n)_{n=1}^{N}\right)=\mathbf{P}\left((\eta_n)_{n=1}^{N}=(x_n)_{n=1}^{N}\:\vline\:S_N>0\right),
\end{equation}
where $\eta$ is an i.i.d.\ sequence of Bernoulli($p$) random variables. Note that the restriction $p<\frac12$ of Subsection \ref{iidsec} is equivalent to taking $\beta_0>0$, and in this regime we will check that $\eta^N$ converges in distribution to $\eta$ as $N\rightarrow\infty$ (see Proposition \ref{infvollim}(a)). We also describe the infinite volume limit in the case $\beta_0\leq 0$ (see Proposition \ref{infvollimhighdens}).
\end{exmp}

\begin{figure}[t]
\begin{center}
\includegraphics[width=0.9\textwidth,height=0.3\textwidth]{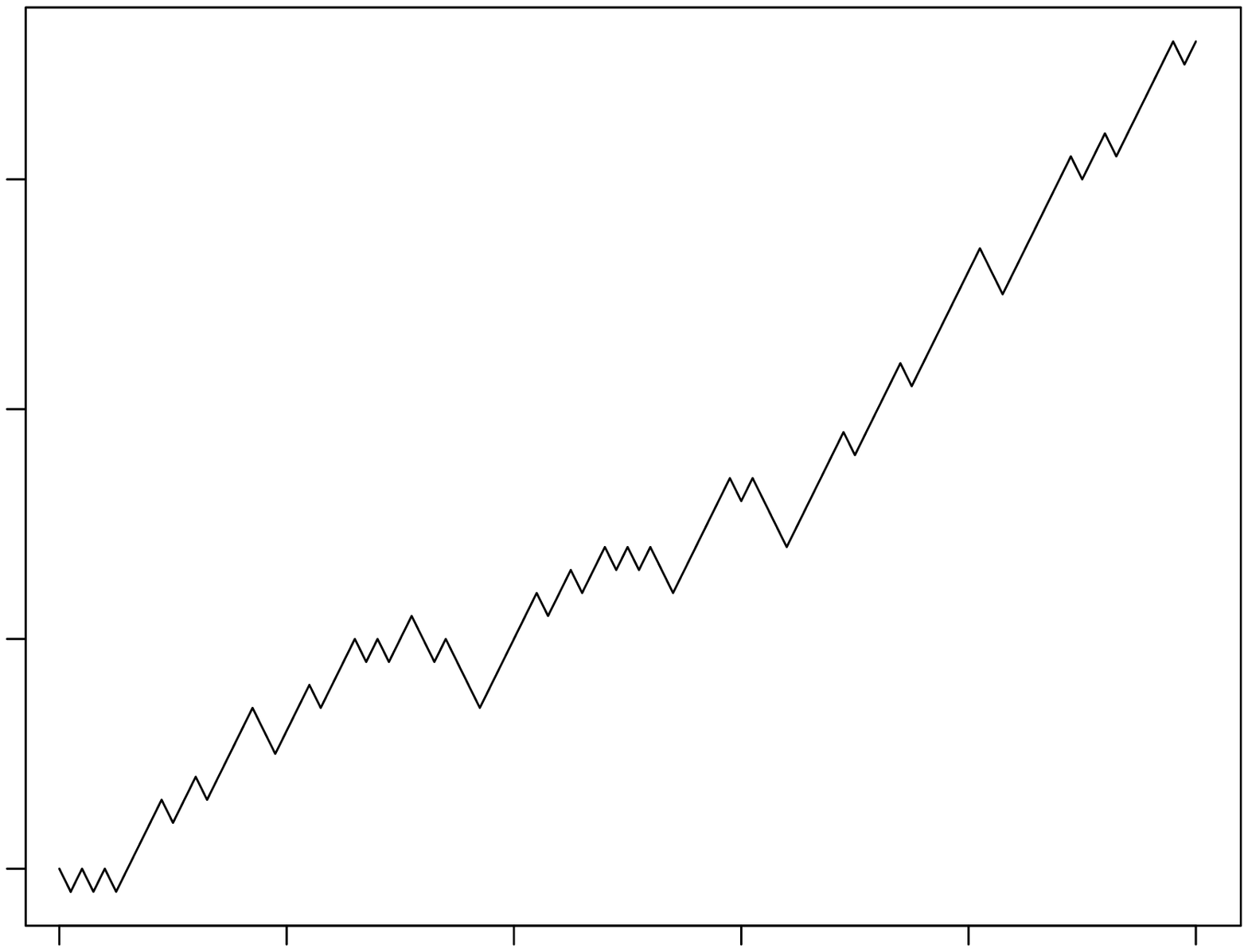}
\rput(-11.6,0.4){$0$}
\rput(-9.53,0.4){$20$}
\rput(-7.46,0.4){$40$}
\rput(-5.39,0.4){$60$}
\rput(-3.32,0.4){$80$}
\rput(-1.25,0.4){$100$}
\rput(-12.3,.8167){$0$}
\rput(-12.33,1.6666){$10$}
\rput(-12.33,2.4833){$20$}
\rput(-12.33,3.3){$30$}

\vspace{-5pt}\BB\BA\BB\BA\BB\BA\BA\BA\BA\BB\BA\BA\BB\BA\BA\BA\BA\BB\BB\BA\BA\BA\BB\BA\BA\BA\BB\BA\BB\BA\BA\BB\BB\BA\BB\BB\BB\BA\BA\BA\BA\BA\BB\BA\BA\BB\BA\BA\BB\BA\BB\BA\BB\BB\BA\BA\BA\BA\BA\BB\BA\BB\BB\BB\BA\BA\BA\BA\BA\BB\BA\BA\BA\BA\BB\BA\BA\BA\BA\BA\BA\BB\BB\BA\BA\BA\BA\BA\BA\BB\BA\BA\BB\BA\BA\BA\BA\BA\BB\BA\\
\vspace{-5pt}\BA\BB\BA\BB\BA\BB\BA\BA\BA\BA\BB\BA\BA\BB\BA\BA\BA\BA\BA\BB\BB\BA\BA\BB\BA\BA\BA\BB\BA\BB\BA\BA\BA\BB\BA\BA\BA\BB\BB\BB\BB\BA\BA\BB\BA\BA\BB\BA\BA\BB\BA\BB\BA\BA\BB\BB\BA\BA\BA\BA\BB\BA\BA\BA\BB\BB\BB\BA\BA\BA\BB\BA\BA\BA\BA\BB\BA\BA\BA\BA\BA\BA\BA\BB\BB\BA\BA\BA\BA\BA\BB\BA\BA\BB\BA\BA\BA\BA\BA\BB\\
\vspace{-5pt}\BB\BA\BB\BA\BB\BA\BB\BA\BA\BA\BA\BB\BA\BA\BB\BA\BA\BA\BA\BA\BA\BB\BB\BA\BB\BA\BA\BA\BB\BA\BB\BA\BA\BA\BB\BA\BA\BA\BA\BA\BA\BB\BB\BA\BB\BB\BA\BB\BB\BA\BB\BA\BB\BA\BA\BA\BB\BB\BA\BA\BA\BB\BA\BA\BA\BA\BA\BB\BB\BB\BA\BB\BA\BA\BA\BA\BB\BA\BA\BA\BA\BA\BA\BA\BA\BB\BB\BA\BA\BA\BA\BB\BA\BA\BB\BA\BA\BA\BA\BA\\
\vspace{-5pt}\BA\BB\BA\BB\BA\BB\BA\BB\BA\BA\BA\BA\BB\BA\BA\BB\BA\BA\BA\BA\BA\BA\BA\BB\BA\BB\BB\BA\BA\BB\BA\BB\BA\BA\BA\BB\BA\BA\BA\BA\BA\BA\BA\BB\BA\BA\BB\BA\BA\BB\BA\BB\BA\BB\BB\BB\BA\BA\BB\BB\BB\BA\BB\BA\BA\BA\BA\BA\BA\BA\BB\BA\BB\BB\BB\BA\BA\BB\BA\BA\BA\BA\BA\BA\BA\BA\BA\BB\BB\BA\BA\BA\BB\BA\BA\BB\BA\BA\BA\BA\\
\vspace{-5pt}\BA\BA\BB\BA\BB\BA\BB\BA\BB\BA\BA\BA\BA\BB\BA\BA\BB\BA\BA\BA\BA\BA\BA\BA\BB\BA\BA\BB\BB\BA\BB\BA\BB\BA\BA\BA\BB\BA\BA\BA\BA\BA\BA\BA\BB\BA\BA\BB\BA\BA\BB\BA\BB\BA\BA\BA\BB\BB\BA\BA\BA\BB\BA\BB\BB\BB\BB\BA\BA\BA\BA\BB\BA\BA\BA\BB\BB\BA\BB\BB\BA\BA\BA\BA\BA\BA\BA\BA\BA\BB\BB\BA\BA\BB\BA\BA\BB\BA\BA\BA\\
\vspace{-5pt}\BA\BA\BA\BB\BA\BB\BA\BB\BA\BB\BA\BA\BA\BA\BB\BA\BA\BB\BA\BA\BA\BA\BA\BA\BA\BB\BA\BA\BA\BB\BA\BB\BA\BB\BB\BA\BA\BB\BA\BA\BA\BA\BA\BA\BA\BB\BA\BA\BB\BA\BA\BB\BA\BB\BA\BA\BA\BA\BB\BB\BA\BA\BB\BA\BA\BA\BA\BB\BB\BB\BB\BA\BB\BA\BA\BA\BA\BB\BA\BA\BB\BB\BB\BA\BA\BA\BA\BA\BA\BA\BA\BB\BB\BA\BB\BA\BA\BB\BA\BA\\
\vspace{-5pt}\BA\BA\BA\BA\BB\BA\BB\BA\BB\BA\BB\BA\BA\BA\BA\BB\BA\BA\BB\BA\BA\BA\BA\BA\BA\BA\BB\BA\BA\BA\BB\BA\BB\BA\BA\BB\BB\BA\BB\BA\BA\BA\BA\BA\BA\BA\BB\BA\BA\BB\BA\BA\BB\BA\BB\BA\BA\BA\BA\BA\BB\BB\BA\BB\BA\BA\BA\BA\BA\BA\BA\BB\BA\BB\BB\BB\BB\BA\BB\BA\BA\BA\BA\BB\BB\BB\BA\BA\BA\BA\BA\BA\BA\BB\BA\BB\BB\BA\BB\BA\\
\vspace{-5pt}\BB\BA\BA\BA\BA\BB\BA\BB\BA\BB\BA\BB\BA\BA\BA\BA\BB\BA\BA\BB\BA\BA\BA\BA\BA\BA\BA\BB\BA\BA\BA\BB\BA\BB\BA\BA\BA\BB\BA\BB\BB\BA\BA\BA\BA\BA\BA\BB\BA\BA\BB\BA\BA\BB\BA\BB\BA\BA\BA\BA\BA\BA\BB\BA\BB\BB\BA\BA\BA\BA\BA\BA\BB\BA\BA\BA\BA\BB\BA\BB\BB\BB\BB\BA\BA\BA\BB\BB\BB\BA\BA\BA\BA\BA\BB\BA\BA\BB\BA\BB\\
\vspace{-5pt}\BA\BB\BB\BA\BA\BA\BB\BA\BB\BA\BB\BA\BB\BA\BA\BA\BA\BB\BA\BA\BB\BA\BA\BA\BA\BA\BA\BA\BB\BA\BA\BA\BB\BA\BB\BA\BA\BA\BB\BA\BA\BB\BB\BA\BA\BA\BA\BA\BB\BA\BA\BB\BA\BA\BB\BA\BB\BA\BA\BA\BA\BA\BA\BB\BA\BA\BB\BB\BA\BA\BA\BA\BA\BB\BA\BA\BA\BA\BB\BA\BA\BA\BA\BB\BB\BB\BA\BA\BA\BB\BB\BB\BB\BA\BA\BB\BA\BA\BB\BA\\
\vspace{-5pt}\BB\BA\BA\BB\BB\BA\BA\BB\BA\BB\BA\BB\BA\BB\BA\BA\BA\BA\BB\BA\BA\BB\BA\BA\BA\BA\BA\BA\BA\BB\BA\BA\BA\BB\BA\BB\BA\BA\BA\BB\BA\BA\BA\BB\BB\BA\BA\BA\BA\BB\BA\BA\BB\BA\BA\BB\BA\BB\BA\BA\BA\BA\BA\BA\BB\BA\BA\BA\BB\BB\BA\BA\BA\BA\BB\BA\BA\BA\BA\BB\BA\BA\BA\BA\BA\BA\BB\BB\BB\BA\BA\BA\BA\BB\BB\BA\BB\BB\BA\BB\\
\vspace{-5pt}\BA\BB\BB\BA\BA\BB\BB\BA\BB\BA\BB\BA\BB\BA\BB\BB\BB\BA\BA\BB\BA\BA\BB\BA\BA\BA\BA\BA\BA\BA\BB\BA\BA\BA\BB\BA\BB\BA\BA\BA\BB\BA\BA\BA\BA\BB\BB\BA\BA\BA\BB\BA\BA\BB\BA\BA\BB\BA\BB\BA\BA\BA\BA\BA\BA\BB\BA\BA\BA\BA\BB\BB\BA\BA\BA\BB\BA\BA\BA\BA\BB\BA\BA\BA\BA\BA\BA\BA\BA\BB\BB\BB\BA\BA\BA\BB\BA\BA\BB\BA\\
\vspace{-5pt}\BA\BA\BA\BB\BB\BA\BA\BB\BA\BB\BA\BB\BA\BB\BA\BA\BA\BB\BB\BA\BB\BB\BA\BB\BB\BA\BA\BA\BA\BA\BA\BB\BA\BA\BA\BB\BA\BB\BA\BA\BA\BB\BA\BA\BA\BA\BA\BB\BB\BA\BA\BB\BA\BA\BB\BA\BA\BB\BA\BB\BA\BA\BA\BA\BA\BA\BB\BA\BA\BA\BA\BA\BB\BB\BA\BA\BB\BA\BA\BA\BA\BB\BA\BA\BA\BA\BA\BA\BA\BA\BA\BA\BB\BB\BB\BA\BB\BA\BA\BB\\
\vspace{-5pt}\BB\BB\BA\BA\BA\BB\BB\BA\BB\BA\BB\BA\BB\BA\BB\BA\BA\BA\BA\BB\BA\BA\BB\BA\BA\BB\BB\BB\BB\BA\BA\BA\BB\BA\BA\BA\BB\BA\BB\BA\BA\BA\BB\BA\BA\BA\BA\BA\BA\BB\BB\BA\BB\BA\BA\BB\BA\BA\BB\BA\BB\BA\BA\BA\BA\BA\BA\BB\BA\BA\BA\BA\BA\BA\BB\BB\BA\BB\BA\BA\BA\BA\BB\BA\BA\BA\BA\BA\BA\BA\BA\BA\BA\BA\BA\BB\BA\BB\BB\BA\\
\vspace{-5pt}\BA\BA\BB\BB\BB\BA\BA\BB\BA\BB\BA\BB\BA\BB\BA\BB\BB\BA\BA\BA\BB\BA\BA\BB\BA\BA\BA\BA\BA\BB\BB\BB\BA\BB\BB\BA\BA\BB\BA\BB\BA\BA\BA\BB\BA\BA\BA\BA\BA\BA\BA\BB\BA\BB\BB\BA\BB\BA\BA\BB\BA\BB\BA\BA\BA\BA\BA\BA\BB\BA\BA\BA\BA\BA\BA\BA\BB\BA\BB\BB\BA\BA\BA\BB\BA\BA\BA\BA\BA\BA\BA\BA\BA\BA\BA\BA\BB\BA\BA\BB\\
\vspace{-5pt}\BB\BA\BA\BA\BA\BB\BB\BA\BB\BA\BB\BA\BB\BA\BB\BA\BA\BB\BB\BB\BA\BB\BA\BA\BB\BA\BA\BA\BA\BA\BA\BA\BB\BA\BA\BB\BB\BA\BB\BA\BB\BB\BB\BA\BB\BA\BA\BA\BA\BA\BA\BA\BB\BA\BA\BB\BA\BB\BB\BA\BB\BA\BB\BA\BA\BA\BA\BA\BA\BB\BA\BA\BA\BA\BA\BA\BA\BB\BA\BA\BB\BB\BA\BA\BB\BA\BA\BA\BA\BA\BA\BA\BA\BA\BA\BA\BA\BB\BA\BA\\
\vspace{-5pt}\BA\BB\BA\BA\BA\BA\BA\BB\BA\BB\BA\BB\BA\BB\BA\BB\BB\BA\BA\BA\BB\BA\BB\BB\BA\BB\BB\BA\BA\BA\BA\BA\BA\BB\BA\BA\BA\BB\BA\BB\BA\BA\BA\BB\BA\BB\BB\BB\BB\BA\BA\BA\BA\BB\BA\BA\BB\BA\BA\BB\BA\BB\BA\BB\BB\BA\BA\BA\BA\BA\BB\BA\BA\BA\BA\BA\BA\BA\BB\BA\BA\BA\BB\BB\BA\BB\BA\BA\BA\BA\BA\BA\BA\BA\BA\BA\BA\BA\BB\BA\\
\vspace{-5pt}\BA\BA\BB\BA\BA\BA\BA\BA\BB\BA\BB\BA\BB\BA\BB\BA\BA\BB\BB\BA\BA\BB\BA\BA\BB\BA\BA\BB\BB\BB\BA\BA\BA\BA\BB\BA\BA\BA\BB\BA\BB\BA\BA\BA\BB\BA\BA\BA\BA\BB\BB\BB\BB\BA\BB\BA\BA\BB\BA\BA\BB\BA\BB\BA\BA\BB\BB\BA\BA\BA\BA\BB\BA\BA\BA\BA\BA\BA\BA\BB\BA\BA\BA\BA\BB\BA\BB\BB\BA\BA\BA\BA\BA\BA\BA\BA\BA\BA\BA\BB\\
\vspace{-5pt}\BB\BA\BA\BB\BA\BA\BA\BA\BA\BB\BA\BB\BA\BB\BA\BB\BA\BA\BA\BB\BB\BA\BB\BA\BA\BB\BA\BA\BA\BA\BB\BB\BB\BA\BA\BB\BA\BA\BA\BB\BA\BB\BA\BA\BA\BB\BA\BA\BA\BA\BA\BA\BA\BB\BA\BB\BB\BA\BB\BB\BA\BB\BA\BB\BB\BA\BA\BB\BB\BA\BA\BA\BB\BA\BA\BA\BA\BA\BA\BA\BB\BA\BA\BA\BA\BB\BA\BA\BB\BB\BA\BA\BA\BA\BA\BA\BA\BA\BA\BA\\
\vspace{-5pt}\BA\BB\BA\BA\BB\BA\BA\BA\BA\BA\BB\BA\BB\BA\BB\BA\BB\BA\BA\BA\BA\BB\BA\BB\BB\BA\BB\BA\BA\BA\BA\BA\BA\BB\BB\BA\BB\BB\BA\BA\BB\BA\BB\BA\BA\BA\BB\BA\BA\BA\BA\BA\BA\BA\BB\BA\BA\BB\BA\BA\BB\BA\BB\BA\BA\BB\BB\BA\BA\BB\BB\BB\BA\BB\BB\BA\BA\BA\BA\BA\BA\BB\BA\BA\BA\BA\BB\BA\BA\BA\BB\BB\BA\BA\BA\BA\BA\BA\BA\BA\\
\vspace{-5pt}\BA\BA\BB\BA\BA\BB\BA\BA\BA\BA\BA\BB\BA\BB\BA\BB\BA\BB\BA\BA\BA\BA\BB\BA\BA\BB\BA\BB\BB\BA\BA\BA\BA\BA\BA\BB\BA\BA\BB\BB\BA\BB\BA\BB\BB\BA\BA\BB\BA\BA\BA\BA\BA\BA\BA\BB\BA\BA\BB\BA\BA\BB\BA\BB\BA\BA\BA\BB\BB\BA\BA\BA\BB\BA\BA\BB\BB\BB\BB\BA\BA\BA\BB\BA\BA\BA\BA\BB\BA\BA\BA\BA\BB\BB\BA\BA\BA\BA\BA\BA\\
\vspace{-5pt}\BA\BA\BA\BB\BA\BA\BB\BA\BA\BA\BA\BA\BB\BA\BB\BA\BB\BA\BB\BA\BA\BA\BA\BB\BA\BA\BB\BA\BA\BB\BB\BA\BA\BA\BA\BA\BB\BA\BA\BA\BB\BA\BB\BA\BA\BB\BB\BA\BB\BB\BA\BA\BA\BA\BA\BA\BB\BA\BA\BB\BA\BA\BB\BA\BB\BA\BA\BA\BA\BB\BB\BA\BA\BB\BA\BA\BA\BA\BA\BB\BB\BB\BA\BB\BB\BA\BA\BA\BB\BA\BA\BA\BA\BA\BB\BB\BA\BA\BA\BA\\
\vspace{-5pt}\BA\BA\BA\BA\BB\BA\BA\BB\BA\BA\BA\BA\BA\BB\BA\BB\BA\BB\BA\BB\BA\BA\BA\BA\BB\BA\BA\BB\BA\BA\BA\BB\BB\BA\BA\BA\BA\BB\BA\BA\BA\BB\BA\BB\BA\BA\BA\BB\BA\BA\BB\BB\BB\BA\BA\BA\BA\BB\BA\BA\BB\BA\BA\BB\BA\BB\BA\BA\BA\BA\BA\BB\BB\BA\BB\BA\BA\BA\BA\BA\BA\BA\BB\BA\BA\BB\BB\BB\BA\BB\BB\BA\BA\BA\BA\BA\BB\BB\BA\BA\\
\vspace{-5pt}\BA\BA\BA\BA\BA\BB\BA\BA\BB\BA\BA\BA\BA\BA\BB\BA\BB\BA\BB\BA\BB\BA\BA\BA\BA\BB\BA\BA\BB\BA\BA\BA\BA\BB\BB\BA\BA\BA\BB\BA\BA\BA\BB\BA\BB\BA\BA\BA\BB\BA\BA\BA\BA\BB\BB\BB\BA\BA\BB\BA\BA\BB\BA\BA\BB\BA\BB\BA\BA\BA\BA\BA\BA\BB\BA\BB\BB\BA\BA\BA\BA\BA\BA\BB\BA\BA\BA\BA\BB\BA\BA\BB\BB\BB\BB\BA\BA\BA\BB\BB\\
\vspace{-5pt}\BB\BB\BB\BA\BA\BA\BB\BA\BA\BB\BA\BA\BA\BA\BA\BB\BA\BB\BA\BB\BA\BB\BA\BA\BA\BA\BB\BA\BA\BB\BA\BA\BA\BA\BA\BB\BB\BA\BA\BB\BA\BA\BA\BB\BA\BB\BA\BA\BA\BB\BA\BA\BA\BA\BA\BA\BB\BB\BA\BB\BB\BA\BB\BA\BA\BB\BA\BB\BA\BA\BA\BA\BA\BA\BB\BA\BA\BB\BB\BA\BA\BA\BA\BA\BB\BA\BA\BA\BA\BB\BA\BA\BA\BA\BA\BB\BB\BB\BA\BA\\
\vspace{-5pt}\BA\BA\BA\BB\BB\BB\BA\BB\BB\BA\BB\BA\BA\BA\BA\BA\BB\BA\BB\BA\BB\BA\BB\BA\BA\BA\BA\BB\BA\BA\BB\BA\BA\BA\BA\BA\BA\BB\BB\BA\BB\BA\BA\BA\BB\BA\BB\BA\BA\BA\BB\BA\BA\BA\BA\BA\BA\BA\BB\BA\BA\BB\BA\BB\BB\BA\BB\BA\BB\BB\BA\BA\BA\BA\BA\BB\BA\BA\BA\BB\BB\BA\BA\BA\BA\BB\BA\BA\BA\BA\BB\BA\BA\BA\BA\BA\BA\BA\BB\BB
\end{center}
\caption{Initial path encoding and first 25 steps of the dynamics for particle configuration sampled from the periodic i.i.d.\ configuration of Example \ref{periidex} with $p=0.35$, i.e.\ $\beta_0=0.62$.}\label{iidfig}
\end{figure}

\begin{exmp}[Periodic Markov initial configuration]\label{permarex} Again similarly to \cite[Remark 1.12]{CKSS}, let $p_0,p_1\in (0,1)$, and consider the parameter choice
\[\beta_0=\log \left(\frac{1-p_0}{p_1}\right),\qquad \beta_1=\log \left(\frac{p_1(1-p_0)}{p_0(1-p_1)}\right),\qquad\beta_k=0,\:\forall k \ge 2.\]
(Figure \ref{markovfig} shows a typical realisation of a configuration chosen according the associated Gibbs measure, and its subsequent evolution.) For these parameters, one can check that
\[\mathbf{P}\left((\eta^N_n)_{n=1}^{N}=(x_n)_{n=1}^{N}\right)\propto\prod_{n=1}^{N}P(x_{n-1},x_{n})\mathbf{1}_{\left\{\sum_{i=1}^{N}{x_i}<N/2\right\}},\]
where, as above, we have supposed that $x_0:=x_N$ in the preceding formula, and the matrix $P=(P(x,y))_{x,y\in\{0,1\}}$ is given by \eqref{pdef}. It follows that one has the following alternative characterisation of the law of $\eta^N$ via the formula
\begin{equation}\label{cyclicmarkov}
\mathbf{E}\left(F\left((\eta^N_n)_{n=1}^{N}\right)\right)=\frac{\mathbf{E}\left(\nu(\eta_0)^{-1}F\left((\eta_n)_{n=1}^{N}\right)\:\vline\:\eta_N=\eta_0,\:S_N>0\right)}{\mathbf{E}\left(\nu(\eta_0)^{-1}\:\vline\:\eta_N=\eta_0,\:S_N>0\right)},
\end{equation}
where $\eta$ is the two-sided stationary Markov configuration of Subsection \ref{markovsec} (noting that we now allow an increased range of parameters $p_0,p_1$), $\nu$ is its invariant measure, and the above formula holds for any function $F:\{0,1\}^{N}\rightarrow\mathbb{R}$. In particular, the initial segment of $\eta^N$ is obtained from $\eta$ by conditioning the latter process to return to its starting state at time $N$ and on seeing less than $N/2$ particles by that time, as well as weighting probabilities by $\nu(\eta_0)^{-1}$. Note that the latter step has the effect of removing the distributional influence of the initial state, thus ensuring the law of $\eta^N$ is stationary under spatial shifts (which is checked more generally as part of Lemma \ref{etalem} below). We note that a similar definition, without the $\nu(\eta_0)^{-1}$ term and $S_N>0$ conditioning, of a (non-stationary) cyclic Markov chain was given in \cite{Albenque}. Finally, the restriction $p_0+p_1<1$ of Subsection \ref{markovsec} is equivalent to taking $\beta_0>0$, and, similarly to the previous example, we will check that $\eta^N$ converges in distribution to $\eta$ as $N\rightarrow\infty$ in this regime (see Proposition \ref{infvollim}(b)).
\end{exmp}

\begin{figure}[t]
\begin{center}
\includegraphics[width=0.9\textwidth,height=0.3\textwidth]{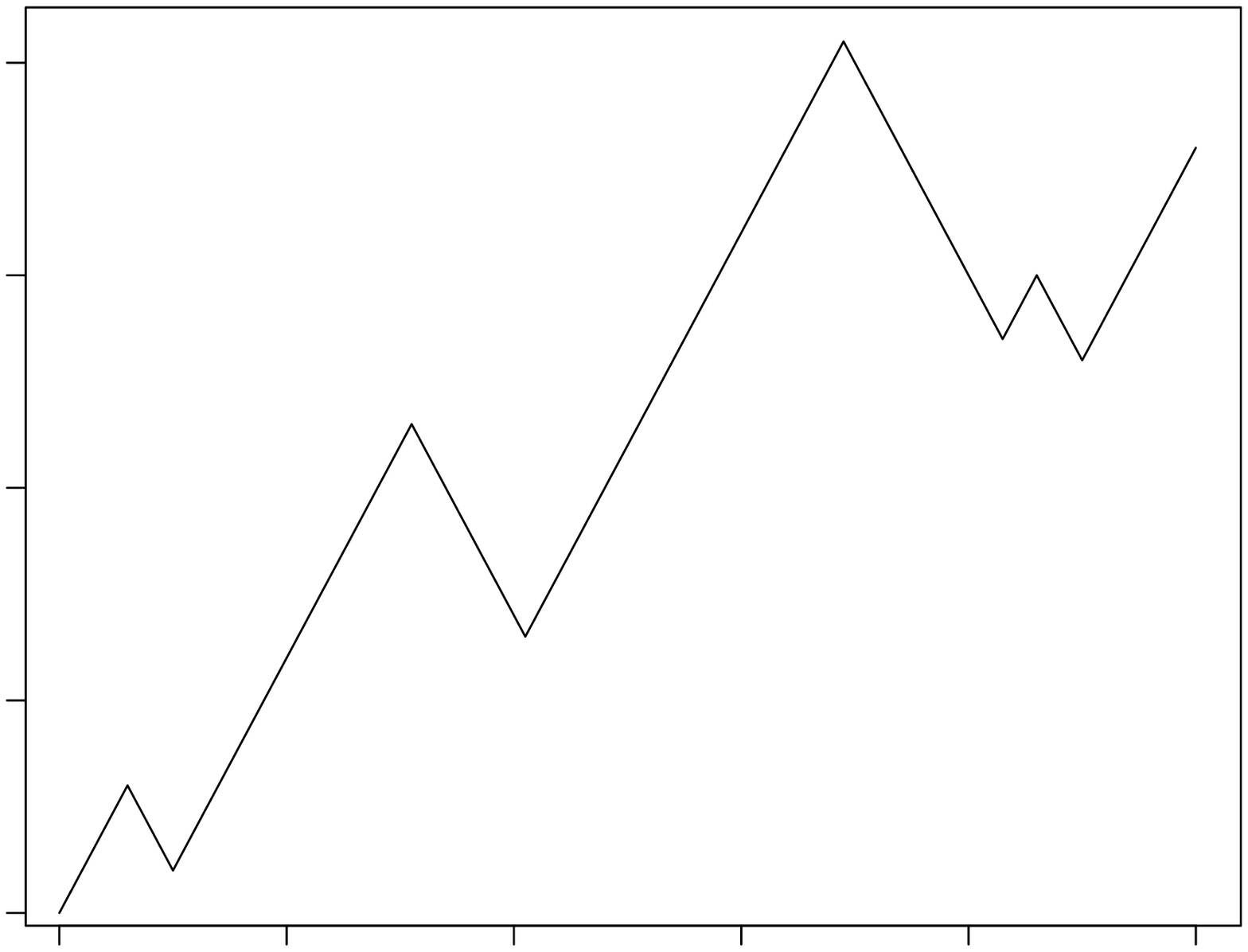}
\rput(-11.6,0.4){$0$}
\rput(-9.53,0.4){$20$}
\rput(-7.46,0.4){$40$}
\rput(-5.39,0.4){$60$}
\rput(-3.32,0.4){$80$}
\rput(-1.25,0.4){$100$}
\rput(-12.3,.7){$0$}
\rput(-12.33,1.45){$10$}
\rput(-12.33,2.2){$20$}
\rput(-12.33,2.95){$30$}
\rput(-12.33,3.7){$40$}

\vspace{-5pt}\BA\BA\BA\BA\BA\BA\BB\BB\BB\BB\BA\BA\BA\BA\BA\BA\BA\BA\BA\BA\BA\BA\BA\BA\BA\BA\BA\BA\BA\BA\BA\BB\BB\BB\BB\BB\BB\BB\BB\BB\BB\BA\BA\BA\BA\BA\BA\BA\BA\BA\BA\BA\BA\BA\BA\BA\BA\BA\BA\BA\BA\BA\BA\BA\BA\BA\BA\BA\BA\BB\BB\BB\BB\BB\BB\BB\BB\BB\BB\BB\BB\BB\BB\BA\BA\BA\BB\BB\BB\BB\BA\BA\BA\BA\BA\BA\BA\BA\BA\BA\\
\vspace{-5pt}\BB\BB\BB\BB\BB\BA\BA\BA\BA\BA\BB\BB\BB\BB\BA\BA\BA\BA\BA\BA\BA\BA\BA\BA\BA\BA\BA\BA\BA\BA\BA\BA\BA\BA\BA\BA\BA\BA\BA\BA\BA\BB\BB\BB\BB\BB\BB\BB\BB\BB\BB\BA\BA\BA\BA\BA\BA\BA\BA\BA\BA\BA\BA\BA\BA\BA\BA\BA\BA\BA\BA\BA\BA\BA\BA\BA\BA\BA\BA\BA\BA\BA\BA\BB\BB\BB\BA\BA\BA\BA\BB\BB\BB\BB\BB\BB\BB\BB\BB\BB\\
\vspace{-5pt}\BA\BA\BA\BA\BA\BB\BB\BB\BB\BB\BA\BA\BA\BA\BB\BB\BB\BB\BB\BB\BB\BB\BB\BB\BB\BB\BB\BB\BA\BA\BA\BA\BA\BA\BA\BA\BA\BA\BA\BA\BA\BA\BA\BA\BA\BA\BA\BA\BA\BA\BA\BB\BB\BB\BB\BB\BB\BB\BB\BB\BB\BA\BA\BA\BA\BA\BA\BA\BA\BA\BA\BA\BA\BA\BA\BA\BA\BA\BA\BA\BA\BA\BA\BA\BA\BA\BB\BB\BB\BA\BA\BA\BA\BA\BA\BA\BA\BA\BA\BA\\
\vspace{-5pt}\BA\BA\BA\BA\BA\BA\BA\BA\BA\BA\BB\BB\BB\BB\BA\BA\BA\BA\BA\BA\BA\BA\BA\BA\BA\BA\BA\BA\BB\BB\BB\BB\BB\BB\BB\BB\BB\BB\BB\BB\BB\BB\BB\BA\BA\BA\BA\BA\BA\BA\BA\BA\BA\BA\BA\BA\BA\BA\BA\BA\BA\BB\BB\BB\BB\BB\BB\BB\BB\BB\BB\BA\BA\BA\BA\BA\BA\BA\BA\BA\BA\BA\BA\BA\BA\BA\BA\BA\BA\BB\BB\BB\BA\BA\BA\BA\BA\BA\BA\BA\\
\vspace{-5pt}\BA\BA\BA\BA\BA\BA\BA\BA\BA\BA\BA\BA\BA\BA\BB\BB\BB\BB\BA\BA\BA\BA\BA\BA\BA\BA\BA\BA\BA\BA\BA\BA\BA\BA\BA\BA\BA\BA\BA\BA\BA\BA\BA\BB\BB\BB\BB\BB\BB\BB\BB\BB\BB\BB\BB\BB\BB\BB\BA\BA\BA\BA\BA\BA\BA\BA\BA\BA\BA\BA\BA\BB\BB\BB\BB\BB\BB\BB\BB\BB\BB\BA\BA\BA\BA\BA\BA\BA\BA\BA\BA\BA\BB\BB\BB\BA\BA\BA\BA\BA\\
\vspace{-5pt}\BA\BA\BA\BA\BA\BA\BA\BA\BA\BA\BA\BA\BA\BA\BA\BA\BA\BA\BB\BB\BB\BB\BA\BA\BA\BA\BA\BA\BA\BA\BA\BA\BA\BA\BA\BA\BA\BA\BA\BA\BA\BA\BA\BA\BA\BA\BA\BA\BA\BA\BA\BA\BA\BA\BA\BA\BA\BA\BB\BB\BB\BB\BB\BB\BB\BB\BB\BB\BB\BB\BB\BA\BA\BA\BA\BA\BA\BA\BA\BA\BA\BB\BB\BB\BB\BB\BB\BB\BB\BB\BB\BB\BA\BA\BA\BB\BB\BB\BB\BA\\
\vspace{-5pt}\BB\BB\BB\BB\BB\BB\BB\BB\BB\BB\BB\BB\BB\BB\BA\BA\BA\BA\BA\BA\BA\BA\BB\BB\BB\BB\BA\BA\BA\BA\BA\BA\BA\BA\BA\BA\BA\BA\BA\BA\BA\BA\BA\BA\BA\BA\BA\BA\BA\BA\BA\BA\BA\BA\BA\BA\BA\BA\BA\BA\BA\BA\BA\BA\BA\BA\BA\BA\BA\BA\BA\BB\BB\BB\BB\BB\BB\BB\BB\BB\BB\BA\BA\BA\BA\BA\BA\BA\BA\BA\BA\BA\BB\BB\BB\BA\BA\BA\BA\BB\\
\vspace{-5pt}\BA\BA\BA\BA\BA\BA\BA\BA\BA\BA\BA\BA\BA\BA\BB\BB\BB\BB\BB\BB\BB\BB\BA\BA\BA\BA\BB\BB\BB\BB\BB\BB\BB\BB\BB\BB\BB\BA\BA\BA\BA\BA\BA\BA\BA\BA\BA\BA\BA\BA\BA\BA\BA\BA\BA\BA\BA\BA\BA\BA\BA\BA\BA\BA\BA\BA\BA\BA\BA\BA\BA\BA\BA\BA\BA\BA\BA\BA\BA\BA\BA\BB\BB\BB\BB\BB\BB\BB\BB\BB\BB\BA\BA\BA\BA\BB\BB\BB\BA\BA\\
\vspace{-5pt}\BB\BB\BB\BB\BB\BB\BB\BA\BA\BA\BA\BA\BA\BA\BA\BA\BA\BA\BA\BA\BA\BA\BB\BB\BB\BB\BA\BA\BA\BA\BA\BA\BA\BA\BA\BA\BA\BB\BB\BB\BB\BB\BB\BB\BB\BB\BB\BB\BB\BB\BB\BB\BA\BA\BA\BA\BA\BA\BA\BA\BA\BA\BA\BA\BA\BA\BA\BA\BA\BA\BA\BA\BA\BA\BA\BA\BA\BA\BA\BA\BA\BA\BA\BA\BA\BA\BA\BA\BA\BA\BA\BB\BB\BB\BB\BA\BA\BA\BB\BB\\
\vspace{-5pt}\BA\BA\BA\BA\BA\BA\BA\BB\BB\BB\BB\BB\BB\BB\BB\BB\BB\BA\BA\BA\BA\BA\BA\BA\BA\BA\BB\BB\BB\BB\BA\BA\BA\BA\BA\BA\BA\BA\BA\BA\BA\BA\BA\BA\BA\BA\BA\BA\BA\BA\BA\BA\BB\BB\BB\BB\BB\BB\BB\BB\BB\BB\BB\BB\BB\BB\BB\BA\BA\BA\BA\BA\BA\BA\BA\BA\BA\BA\BA\BA\BA\BA\BA\BA\BA\BA\BA\BA\BA\BA\BA\BA\BA\BA\BA\BB\BB\BB\BA\BA\\
\vspace{-5pt}\BB\BA\BA\BA\BA\BA\BA\BA\BA\BA\BA\BA\BA\BA\BA\BA\BA\BB\BB\BB\BB\BB\BB\BB\BB\BB\BA\BA\BA\BA\BB\BB\BB\BB\BB\BA\BA\BA\BA\BA\BA\BA\BA\BA\BA\BA\BA\BA\BA\BA\BA\BA\BA\BA\BA\BA\BA\BA\BA\BA\BA\BA\BA\BA\BA\BA\BA\BB\BB\BB\BB\BB\BB\BB\BB\BB\BB\BB\BB\BB\BB\BB\BA\BA\BA\BA\BA\BA\BA\BA\BA\BA\BA\BA\BA\BA\BA\BA\BB\BB\\
\vspace{-5pt}\BA\BB\BB\BB\BA\BA\BA\BA\BA\BA\BA\BA\BA\BA\BA\BA\BA\BA\BA\BA\BA\BA\BA\BA\BA\BA\BB\BB\BB\BB\BA\BA\BA\BA\BA\BB\BB\BB\BB\BB\BB\BB\BB\BB\BB\BA\BA\BA\BA\BA\BA\BA\BA\BA\BA\BA\BA\BA\BA\BA\BA\BA\BA\BA\BA\BA\BA\BA\BA\BA\BA\BA\BA\BA\BA\BA\BA\BA\BA\BA\BA\BA\BB\BB\BB\BB\BB\BB\BB\BB\BB\BB\BB\BB\BB\BB\BB\BA\BA\BA\\
\vspace{-5pt}\BB\BA\BA\BA\BB\BB\BB\BB\BB\BB\BB\BB\BB\BB\BB\BB\BB\BB\BA\BA\BA\BA\BA\BA\BA\BA\BA\BA\BA\BA\BB\BB\BB\BB\BA\BA\BA\BA\BA\BA\BA\BA\BA\BA\BA\BB\BB\BB\BB\BB\BB\BB\BB\BB\BB\BA\BA\BA\BA\BA\BA\BA\BA\BA\BA\BA\BA\BA\BA\BA\BA\BA\BA\BA\BA\BA\BA\BA\BA\BA\BA\BA\BA\BA\BA\BA\BA\BA\BA\BA\BA\BA\BA\BA\BA\BA\BA\BB\BB\BB\\
\vspace{-5pt}\BA\BB\BB\BB\BA\BA\BA\BA\BA\BA\BA\BA\BA\BA\BA\BA\BA\BA\BB\BB\BB\BB\BB\BB\BB\BB\BB\BB\BB\BB\BA\BA\BA\BA\BB\BB\BB\BB\BB\BB\BB\BA\BA\BA\BA\BA\BA\BA\BA\BA\BA\BA\BA\BA\BA\BB\BB\BB\BB\BB\BB\BB\BB\BB\BB\BA\BA\BA\BA\BA\BA\BA\BA\BA\BA\BA\BA\BA\BA\BA\BA\BA\BA\BA\BA\BA\BA\BA\BA\BA\BA\BA\BA\BA\BA\BA\BA\BA\BA\BA\\
\vspace{-5pt}\BA\BA\BA\BA\BB\BB\BB\BA\BA\BA\BA\BA\BA\BA\BA\BA\BA\BA\BA\BA\BA\BA\BA\BA\BA\BA\BA\BA\BA\BA\BB\BB\BB\BB\BA\BA\BA\BA\BA\BA\BA\BB\BB\BB\BB\BB\BB\BB\BB\BB\BB\BB\BB\BB\BB\BA\BA\BA\BA\BA\BA\BA\BA\BA\BA\BB\BB\BB\BB\BB\BB\BB\BB\BB\BB\BB\BA\BA\BA\BA\BA\BA\BA\BA\BA\BA\BA\BA\BA\BA\BA\BA\BA\BA\BA\BA\BA\BA\BA\BA\\
\vspace{-5pt}\BA\BA\BA\BA\BA\BA\BA\BB\BB\BB\BA\BA\BA\BA\BA\BA\BA\BA\BA\BA\BA\BA\BA\BA\BA\BA\BA\BA\BA\BA\BA\BA\BA\BA\BB\BB\BB\BB\BA\BA\BA\BA\BA\BA\BA\BA\BA\BA\BA\BA\BA\BA\BA\BA\BA\BB\BB\BB\BB\BB\BB\BB\BB\BB\BB\BA\BA\BA\BA\BA\BA\BA\BA\BA\BA\BA\BB\BB\BB\BB\BB\BB\BB\BB\BB\BB\BB\BB\BB\BB\BB\BA\BA\BA\BA\BA\BA\BA\BA\BA\\
\vspace{-5pt}\BB\BB\BB\BB\BB\BB\BA\BA\BA\BA\BB\BB\BB\BA\BA\BA\BA\BA\BA\BA\BA\BA\BA\BA\BA\BA\BA\BA\BA\BA\BA\BA\BA\BA\BA\BA\BA\BA\BB\BB\BB\BB\BA\BA\BA\BA\BA\BA\BA\BA\BA\BA\BA\BA\BA\BA\BA\BA\BA\BA\BA\BA\BA\BA\BA\BB\BB\BB\BB\BB\BB\BB\BB\BB\BB\BA\BA\BA\BA\BA\BA\BA\BA\BA\BA\BA\BA\BA\BA\BA\BA\BB\BB\BB\BB\BB\BB\BB\BB\BB\\
\vspace{-5pt}\BA\BA\BA\BA\BA\BA\BB\BB\BB\BB\BA\BA\BA\BB\BB\BB\BB\BB\BB\BB\BB\BB\BB\BB\BB\BB\BB\BA\BA\BA\BA\BA\BA\BA\BA\BA\BA\BA\BA\BA\BA\BA\BB\BB\BB\BB\BA\BA\BA\BA\BA\BA\BA\BA\BA\BA\BA\BA\BA\BA\BA\BA\BA\BA\BA\BA\BA\BA\BA\BA\BA\BA\BA\BA\BA\BB\BB\BB\BB\BB\BB\BB\BB\BB\BB\BA\BA\BA\BA\BA\BA\BA\BA\BA\BA\BA\BA\BA\BA\BA\\
\vspace{-5pt}\BA\BA\BA\BA\BA\BA\BA\BA\BA\BA\BB\BB\BB\BA\BA\BA\BA\BA\BA\BA\BA\BA\BA\BA\BA\BA\BA\BB\BB\BB\BB\BB\BB\BB\BB\BB\BB\BB\BB\BB\BB\BB\BA\BA\BA\BA\BB\BB\BB\BB\BA\BA\BA\BA\BA\BA\BA\BA\BA\BA\BA\BA\BA\BA\BA\BA\BA\BA\BA\BA\BA\BA\BA\BA\BA\BA\BA\BA\BA\BA\BA\BA\BA\BA\BA\BB\BB\BB\BB\BB\BB\BB\BB\BB\BB\BA\BA\BA\BA\BA\\
\vspace{-5pt}\BB\BB\BB\BB\BB\BA\BA\BA\BA\BA\BA\BA\BA\BB\BB\BB\BA\BA\BA\BA\BA\BA\BA\BA\BA\BA\BA\BA\BA\BA\BA\BA\BA\BA\BA\BA\BA\BA\BA\BA\BA\BA\BB\BB\BB\BB\BA\BA\BA\BA\BB\BB\BB\BB\BB\BB\BB\BB\BB\BB\BB\BB\BB\BB\BB\BA\BA\BA\BA\BA\BA\BA\BA\BA\BA\BA\BA\BA\BA\BA\BA\BA\BA\BA\BA\BA\BA\BA\BA\BA\BA\BA\BA\BA\BA\BB\BB\BB\BB\BB\\
\vspace{-5pt}\BA\BA\BA\BA\BA\BB\BB\BB\BB\BB\BB\BB\BB\BA\BA\BA\BB\BB\BB\BB\BB\BA\BA\BA\BA\BA\BA\BA\BA\BA\BA\BA\BA\BA\BA\BA\BA\BA\BA\BA\BA\BA\BA\BA\BA\BA\BB\BB\BB\BB\BA\BA\BA\BA\BA\BA\BA\BA\BA\BA\BA\BA\BA\BA\BA\BB\BB\BB\BB\BB\BB\BB\BB\BB\BB\BB\BB\BB\BB\BB\BA\BA\BA\BA\BA\BA\BA\BA\BA\BA\BA\BA\BA\BA\BA\BA\BA\BA\BA\BA\\
\vspace{-5pt}\BA\BA\BA\BA\BA\BA\BA\BA\BA\BA\BA\BA\BA\BB\BB\BB\BA\BA\BA\BA\BA\BB\BB\BB\BB\BB\BB\BB\BB\BB\BB\BA\BA\BA\BA\BA\BA\BA\BA\BA\BA\BA\BA\BA\BA\BA\BA\BA\BA\BA\BB\BB\BB\BB\BA\BA\BA\BA\BA\BA\BA\BA\BA\BA\BA\BA\BA\BA\BA\BA\BA\BA\BA\BA\BA\BA\BA\BA\BA\BA\BB\BB\BB\BB\BB\BB\BB\BB\BB\BB\BB\BB\BB\BB\BB\BA\BA\BA\BA\BA\\
\vspace{-5pt}\BB\BB\BB\BB\BB\BB\BB\BB\BB\BB\BA\BA\BA\BA\BA\BA\BB\BB\BB\BA\BA\BA\BA\BA\BA\BA\BA\BA\BA\BA\BA\BB\BB\BB\BB\BB\BB\BB\BB\BB\BB\BA\BA\BA\BA\BA\BA\BA\BA\BA\BA\BA\BA\BA\BB\BB\BB\BB\BA\BA\BA\BA\BA\BA\BA\BA\BA\BA\BA\BA\BA\BA\BA\BA\BA\BA\BA\BA\BA\BA\BA\BA\BA\BA\BA\BA\BA\BA\BA\BA\BA\BA\BA\BA\BA\BB\BB\BB\BB\BB\\
\vspace{-5pt}\BA\BA\BA\BA\BA\BA\BA\BA\BA\BA\BB\BB\BB\BB\BB\BB\BA\BA\BA\BB\BB\BB\BB\BB\BB\BB\BB\BB\BB\BB\BB\BA\BA\BA\BA\BA\BA\BA\BA\BA\BA\BB\BB\BB\BB\BB\BB\BB\BB\BB\BB\BA\BA\BA\BA\BA\BA\BA\BB\BB\BB\BB\BA\BA\BA\BA\BA\BA\BA\BA\BA\BA\BA\BA\BA\BA\BA\BA\BA\BA\BA\BA\BA\BA\BA\BA\BA\BA\BA\BA\BA\BA\BA\BA\BA\BA\BA\BA\BA\BA\\
\vspace{-5pt}\BA\BA\BA\BA\BA\BA\BA\BA\BA\BA\BA\BA\BA\BA\BA\BA\BB\BB\BB\BA\BA\BA\BA\BA\BA\BA\BA\BA\BA\BA\BA\BB\BB\BB\BB\BB\BB\BB\BB\BB\BB\BA\BA\BA\BA\BA\BA\BA\BA\BA\BA\BB\BB\BB\BB\BB\BB\BB\BA\BA\BA\BA\BB\BB\BB\BB\BB\BB\BB\BB\BB\BB\BB\BB\BA\BA\BA\BA\BA\BA\BA\BA\BA\BA\BA\BA\BA\BA\BA\BA\BA\BA\BA\BA\BA\BA\BA\BA\BA\BA
\end{center}
\caption{Initial path encoding and first 25 steps of the dynamics for particle configuration sampled from the periodic Markov configuration of Example \ref{permarex} with $N=100$ and parameters $p_0=0.11$, $p_1=0.80$, i.e.\ $\beta_0=0.11$, $\beta_1=3.48$. (Note these parameters correspond to a density of $\rho=0.35$ for the non-periodic version of the configuration, matching that for the non-periodic version of the i.i.d.\ example shown in Figure \ref{iidfig}.)}\label{markovfig}
\end{figure}

\begin{exmp}[Periodic bounded soliton configuration]\label{perboundedex} Once again similarly to \cite[Remark 1.12]{CKSS}, let $p\in (0,1)$ and $K\in\mathbb{N}$, and consider the parameter choice
\[\beta_0=\log \left(\frac{1-p}{p}\right),\qquad \beta_k=0,\:\forall k\in\{1,\dots,K\},\qquad \beta_k=\infty,\:\forall k> K.\]
For these parameters, one can check that
\begin{equation}\label{bexp0}
\mathbf{E}\left(F\left((\eta^N_n)_{n=1}^{N}\right)\right)=\mathbf{E}\left(F\left((\eta_n)_{n=1}^{N}\right)\:\vline\:\mathcal{A}_{N,K},\:S_N>0\right),
\end{equation}
where $\eta$ is an i.i.d.\ sequence of Bernoulli($p$) random variables, and
\begin{equation}\label{ankdef}
\mathcal{A}_{N,K}=\left\{\max_{0\leq n\leq N}\max\left\{\max_{0\leq m\leq n}(S_m-S_n),\:\max_{n\leq m\leq N}(S_m-S_N-S_n)\right\}\leq K\right\}.
\end{equation}
(Note the expression involving nested maxima simply describes the supremum of the carrier corresponding to the cyclic repetition of $(\eta_n)_{n=1}^{N}$.) We will check that, for any parameters $p\in (0,1)$ and $K\in\mathbb{N}$, $\eta^N$ converges in distribution to $\tilde{\eta}^{(K)}$, the example of Subsection \ref{boundedsec}, as $N\rightarrow\infty$ (see Proposition \ref{infvollim}(c)).
\end{exmp}

We now give the infinite volume limits for the previous three examples.

\begin{prop}\label{infvollim} (a) Let $p\in(0,\frac12)$, and $\eta^{N,iid}$ be the periodic i.i.d.\ configuration of Example \ref{periidex} (i.e.\ with law given by \eqref{iidexp}). Then
\[\eta^{N,iid}\buildrel{d}\over\rightarrow\eta^{iid}\]
as $N\rightarrow\infty$, where $\eta^{iid}$ is the i.i.d.\ configuration of Subsection \ref{iidsec}.\\
(b) Let $p_0,p_1\in(0,1)$ be such that $p_0+p_1<1$, and $\eta^{N,Mar}$ be the periodic Markov configuration of Example \ref{periidex} (i.e.\ with law given by \eqref{cyclicmarkov}). Then
\[\eta^{N,Mar}\buildrel{d}\over\rightarrow\eta^{Mar}\]
as $N\rightarrow\infty$, where $\eta^{Mar}$ is the configuration of Subsection \ref{markovsec}.\\
(c) Let $p\in(0,1)$ and $K\in\mathbb{N}$, and $\eta^{N,b}$ be the periodic bounded soliton configuration of Example \ref{perboundedex} (i.e.\ with law given by \eqref{bexp0}). Then
\[\eta^{N,b}\buildrel{d}\over\rightarrow\tilde{\eta}^{(K)}\]
as $N\rightarrow\infty$, where $\tilde{\eta}^{(K)}$ is the bounded soliton example of Subsection \ref{boundedsec}.
\end{prop}
\begin{proof} The proof of (a) is straightforward. Indeed, starting from \eqref{iidexp}, and applying that $\mathbf{P}(S_N>0)\rightarrow 1$, we obtain: for any $M\in\mathbb{N}$, $x\in\{0,1\}^M$,
\[\mathbf{P}\left((\eta^{N,iid}_n)_{n=1}^{M}=(x_n)_{n=1}^{M}\right)=\frac{\mathbf{P}\left((\eta^{iid}_n)_{n=1}^{M}=(x_n)_{n=1}^{M},\:S_N>0\right)}{\mathbf{P}\left(S_N>0\right)}\rightarrow\mathbf{P}\left((\eta^{iid}_n)_{n=1}^{M}=(x_n)_{n=1}^{M}\right).\]

For (b), we start from \eqref{cyclicmarkov} to deduce: for any $M\in\mathbb{N}$, $x\in\{0,1\}^M$,
\begin{equation}\label{ratio}
\mathbf{P}\left((\eta^{N,Mar}_n)_{n=1}^{M}=(x_n)_{n=1}^{M}\right)=
\frac{\mathbf{E}\left(\nu(\eta^{Mar}_0)^{-1}\mathbf{1}_{\left\{(\eta^{Mar}_n)_{n=1}^{M}=(x_n)_{n=1}^{M},\:\eta^{Mar}_N=\eta^{Mar}_0,\:S_N>0\right\}}\right)}
{\mathbf{E}\left(\nu(\eta^{Mar}_0)^{-1}\mathbf{1}_{\left\{\:\eta^{Mar}_N=\eta^{Mar}_0,\:S_N>0\right\}}\right)}.
\end{equation}
Now, by the definition of the Markov chain, the numerator can be written
\[\sum_{x_0\in\{0,1\}}\prod_{n=1}^{M}P(x_{n-1},x_n)\mathbf{P}\left(\eta^{Mar}_{N-M}=x_0,\:S_{N-M}+\sum_{n=1}^M(1-2x_n)>0\:\vline\:\eta^{Mar}_0=x_M\right).\]
Since $N^{-1}S_N\rightarrow1-2\rho>0$, $\mathbf{P}$-a.s., where $\rho$ was defined at \eqref{rhodef}, it readily follows that this expression converges as $N\rightarrow\infty$ to
\[\sum_{x_0\in\{0,1\}}\nu(x_0)\prod_{n=1}^{M}P(x_{n-1},x_n)=\nu(x_1)\prod_{n=2}^{M}P(x_{n-1},x_n)=\mathbf{P}\left((\eta^{Mar}_n)_{n=1}^{M}=(x_n)_{n=1}^{M}\right).\]
Summing over $x\in\{0,1\}^M$ shows that the denominator of \eqref{ratio} converges to one, and hence we have established the result in this case.

Finally, we prove (c) for $p\in(0,1)$, $K\in \mathbb{N}$. To this end, we first provide an alternative characterisation of \eqref{bexp0}. In particular, let $(\eta_n)_{n\geq 1}$ be i.i.d.\ with parameter $p\in(0,1)$, and $\hat{W}$ be the associated carrier process started from the initial condition that $\hat{W}_0$ is uniform on $\{0,1,\dots,K\}$. (Note the latter process is a Markov chain on $\mathbb{Z}_+$ with transition probabilities as at \eqref{Wprobs}.) We then claim that
\begin{equation}\label{bexp}
\mathbf{E}\left(F\left((\eta^{N,b}_n)_{n=1}^{N}\right)\right)=
\mathbf{E}\left(F\left((\eta_n)_{n=1}^{N}\right)\:\vline\:\hat{W}_N=\hat{W}_0,\:\max_{1\leq n\leq N}\hat{W}_n\leq K,\:S_N>0\right).
\end{equation}
To prove this, observe that for any sequence $x\in\{0,1\}^N$
\begin{eqnarray*}
\lefteqn{\mathbf{P}\left((\eta_n)_{n=1}^{N}=(x_n)_{n=1}^{N}\:\vline\:\hat{W}_N=\hat{W}_0,\:\max_{1\leq n\leq N}\hat{W}_n\leq K,\:S_N>0\right)}\\
&=&c
\sum_{w_0=0}^K\mathbf{P}\left((\eta_n)_{n=1}^{N}=(x_n)_{n=1}^{N},\:\hat{W}_0=w_0\right)\mathbf{1}_{\{w^{x,w_0}_N=w_0,\:
\max_{1\leq n\leq N}w^{x,w_0}_n\leq K,\:\sum_{n=1}^Nx_n<N/2\}},
\end{eqnarray*}
where $c:=\mathbf{P}(\hat{W}_N=\hat{W}_0,\:\max_{1\leq n\leq N}\hat{W}_n\leq K,\:S_N>0)^{-1}$ is the required normalising constant, and $(w^{x,w_0}_n)_{n=1}^N$ is the path of the carrier process corresponding to initial carrier value $w_0$ and particle configuration $x$. Since we are assuming the initial distribution of $\hat{W}$ is uniform, and it also holds that $\hat{W}_0$ is independent of $(\eta_n)_{n=1}^{N}=(x_n)_{n=1}^{N}$, we thus have that the above expression is equal to
\[c(K+1)^{-1}\mathbf{P}\left((\eta_n)_{n=1}^{N}=(x_n)_{n=1}^{N}\right)\mathbf{1}_{\{\sum_{n=1}^Nx_n<N/2\}}
\sum_{w_0=0}^K\mathbf{1}_{\{w^{x,w_0}_N=w_0,\:
\max_{1\leq n\leq N}w^{x,w_0}_n\leq K\}}.\]
Now, under the conditions that $w^{x,w_0}_N=w_0$ and $\sum_{n=1}^Nx_n<N/2$, it is straightforward to check that $\max_{1\leq n\leq N}w^{x,w_0}_n\leq K$ is equivalent to $x\in\mathcal{A}_{N,K}$ (in the sense that the associated path encoding satisfies the condition given in the definition of $\mathcal{A}_{N,K}$ at \eqref{ankdef}). And, it is moreover possible to show that under $\sum_{n=1}^Nx_n<N/2$ and $x\in\mathcal{A}_{N,K}$, the condition $w^{x,w_0}_N=w_0$ holds for exactly one $w_0$ (corresponding to $\max_{0\leq n\leq N}S_n-S_N$ for the relevant path encoding). Hence we conclude that
\begin{eqnarray*}
\lefteqn{\mathbf{P}\left((\eta_n)_{n=1}^{N}=(x_n)_{n=1}^{N}\:\vline\:\hat{W}_N=\hat{W}_0,\:\max_{1\leq n\leq N}\hat{W}_n\leq K,\:S_N>0\right)}\\
&=&c(K+1)^{-1}\mathbf{P}\left((\eta_n)_{n=1}^{N}=(x_n)_{n=1}^{N}\right)\mathbf{1}_{\{x\in\mathcal{A}_{N,K},\:
\sum_{n=1}^Nx_n<N/2\}},
\end{eqnarray*}
and hence \eqref{bexp} follows from the characterisation of the law of $\eta^{N,b}$ at \eqref{bexp0}. To study the limit of \eqref{bexp} as $N\rightarrow\infty$, we start by considering the corresponding formula without the $S_N>0$ conditioning. That is, given a sequence $x\in\{0,1\}^M$ representing a particle configuration, we will deduce the $N\rightarrow\infty$ asymptotics of
\begin{equation}\label{bbb}
\mathbf{P}\left((\eta_n)_{n=1}^{M}=(x_n)_{n=1}^{M}\:\vline\:\hat{W}_N=\hat{W}_0,\:\max_{1\leq n\leq N}\hat{W}_n\leq K\right).
\end{equation}
Decomposing over the value of $\hat{W}_0$, we have that the above probability can be written
\begin{eqnarray*}
\lefteqn{\sum_{w_0=0}^K\mathbf{P}\left(\hat{W}_0=w_0,\:(\hat{W}_n)_{n=1}^M=(w^{x,w_0}_n)_{n=1}^M\:\vline\:\hat{W}_N=\hat{W}_0,\:\max_{1\leq n\leq N}\hat{W}_n\leq K\right)}\\
&=&\frac{\sum_{w_0=0}^K\prod_{n=1}^MP_W(w^{x,w_0}_{n-1},w^{x,w_0}_{n})
\mathbf{P}\left(\hat{W}_{N-M}=w_0,\:\max_{1\leq n\leq {N-M}}\hat{W}_n\leq K\:\vline\:\hat{W}_{0}=w^{x,w_0}_M\right)}{(K+1)\mathbf{P}\left(\hat{W}_N=\hat{W}_0,\:\max_{1\leq n\leq N}\hat{W}_n\leq K\right)},
\end{eqnarray*}
where $P_W$ is the transition matrix of $\hat{W}$, as given by \eqref{Wprobs}. Similarly decomposing the numerator, this equals
\begin{equation}\label{aaa}
\frac{\sum_{w_0=0}^K\prod_{n=1}^MP_W(w^{x,w_0}_{n-1},w^{x,w_0}_{n})
\mathbf{P}\left(\hat{W}_{N-M}=w_0,\:\max_{1\leq n\leq {N-M}}\hat{W}_n\leq K\:\vline\:\hat{W}_{0}=w^{x,w_0}_M\right)}{\sum_{w_0=0}^K\mathbf{P}\left(\hat{W}_N=w_0,\:\max_{1\leq n\leq N}\hat{W}_n\leq K\:\vline\:\hat{W}_{0}=w_0\right)}.
\end{equation}
Now, applying \cite[Proposition 1]{GT}, we have that
\[\mathbf{P}\left(\hat{W}_{N-M}=w_0,\:\max_{1\leq n\leq {N-M}}\hat{W}_n\leq K\:\vline\:\hat{W}_{0}=w^{x,w_0}_M\right)\sim \frac{\lambda_K^{N-M}h_K(w^{x,w_0}_M)\tilde{\pi}^{(K)}_{w_0}}{h_K(w_0)},\]
where we have applied the notation of Subsection \ref{boundedsec}, and similarly
\[\mathbf{P}\left(\hat{W}_{N}=w_0,\:\max_{1\leq n\leq {N}}\hat{W}_n\leq K\:\vline\:\hat{W}_{0}=w_0\right)\sim
\lambda_K^{N}\tilde{\pi}^{(K)}_{w_0}.\]
It follows that \eqref{aaa} converges as $N\rightarrow\infty$ to
\begin{eqnarray*}
\lefteqn{\sum_{w_0=0}^K\tilde{\pi}^{(K)}_{w_0}\lambda_K^{-M}h_K(w^{x,w_0}_M)h_K(w_0)^{-1}\prod_{n=1}^MP_W(w^{x,w_0}_{n-1},w^{x,w_0}_{n})}\\
&=&\sum_{w_0\in\mathbb{Z}_+}\tilde{\pi}^{(K)}_{w_0}\prod_{n=1}^M\tilde{P}^{(K)}(w^{x,w_0}_{n-1},w^{x,w_0}_{n})\\
&=&\mathbf{P}\left(\left(\tilde{W}^{(K)}_n\right)_{n=1}^M=\left(w^{x,\tilde{W}^{(K)}_0}_n\right)_{n=1}^M\right)\\
&=&\mathbf{P}\left(\left(\tilde{\eta}^{(K)}_n\right)_{n=1}^M=\left(x_n\right)_{n=1}^M\right).
\end{eqnarray*}
In order to complete the proof, we need to show the same limit when the $S_N>0$ conditioning is reintroduced. To this end, first suppose $\tilde{\eta}^{N,b}$ is a random configuration chosen such that $\mathbf{P}((\tilde{\eta}^{N,b}_n)_{n=1}^N=(x_n)_{n=1}^N)$ is given by \eqref{bbb} (with $M=N$), so that ${\eta}^{N,b}$ has the law of $\tilde{\eta}^{N,b}$ conditioned on $\sum_{n=1}^N(1-2\tilde{\eta}^{N,b}_n)>0$. Moreover, observe that, for any $M\in\mathbb{N}$,
\begin{eqnarray*}
\limsup_{N\rightarrow\infty}\mathbf{P}\left(\sum_{n=1}^N(1-2\tilde{\eta}^{N,b}_n)\leq 0\right)&\leq & \limsup_{N\rightarrow\infty}\mathbf{P}\left(\sum_{n=1}^M(1-2\tilde{\eta}^{N,b}_n)
\leq K\right)\\
&=&\mathbf{P}\left(\sum_{n=1}^M(1-2\tilde{\eta}^{(K)}_n)\leq K\right),
\end{eqnarray*}
and, by Corollary \ref{boundedcor}, the final expression here can be made arbitrarily small by choosing $M$ large. Hence, in conjunction with the previous part of the proof, we obtain that
\begin{eqnarray*}
\mathbf{P}\left(\left({\eta}^{N,b}_n\right)_{n=1}^M=\left(x_n\right)_{n=1}^M\right)
&=&\mathbf{P}\left(\left(\tilde{\eta}^{N,b}_n\right)_{n=1}^M=\left(x_n\right)_{n=1}^M\:\vline\:\sum_{n=1}^N(1-2\tilde{\eta}^{N,b}_n)>0\right)\\
&\rightarrow&\mathbf{P}\left(\left(\tilde{\eta}^{(K)}_n\right)_{n=1}^M=\left(x_n\right)_{n=1}^M\right),
\end{eqnarray*}
as desired.
\end{proof}

In the final result of this section, we demonstrate that if we take the infinite volume limit in the periodic i.i.d.\ initial configuration (Example \ref{periidex}) for a parameter $\beta_0\leq0$ (corresponding to $p\geq \frac12$), then the limit is independent of the particular parameter chosen, being equal to the configuration consisting of i.i.d.\ Bernoulli parameter $\frac12$ random variables. Note that, whilst the latter configuration can be thought of as lying on the boundary of a collection of random configurations that are invariant for $T$, the two-sided dynamics are not even defined in this case (since obviously $M_0=\infty$). Moreover, we observe that its density is critical, in the sense that any infinite volume limit of a periodic Gibbs measure can be no greater than $\frac12$. Whilst we do not pursue this point further, we expect similar phenomena for other choices of parameter $(\beta_k)_{k\geq 0}$ that, beyond the $S_N>0$ restriction, favour configurations of density greater than or equal to $\frac12$.

\begin{prop}\label{infvollimhighdens} Let $p\geq \frac12$, and $\eta^{N,iid}$ be the periodic i.i.d.\ configuration of Example \ref{periidex} (i.e.\ with law given by \eqref{iidexp}). Then
\[\eta^{N,iid}\buildrel{d}\over\rightarrow\eta^{\frac12}\]
as $N\rightarrow\infty$, where $\eta^{\frac12}$ is the i.i.d.\ configuration of Subsection \ref{iidsec} with $p=\frac12$.
\end{prop}
\begin{proof} We first deal with the case when $p=\frac12$. For this parameter choice, we have that
\begin{eqnarray*}
\mathbf{P}\left((\eta^{N,iid}_n)_{n=1}^{M}=(x_n)_{n=1}^{M}\right)&=&\frac{\mathbf{P}\left((\eta^{\frac12}_n)_{n=1}^{M}=(x_n)_{n=1}^{M}\right)\mathbf{P}\left(S_{N-M}+\sum_{n=1}^M(1-2x_n)>0\right)}{\mathbf{P}\left(S_N>0\right)}\\
&\rightarrow&\mathbf{P}\left((\eta^{\frac12}_n)_{n=1}^{M}=(x_n)_{n=1}^{M}\right),
\end{eqnarray*}
where in the above $S$ is the path encoding of $\eta^\frac12$, and the limit is a ready consequence of the fact that ${N}^{-1/2}S_N$ converges in distribution to a standard normal as $N\rightarrow\infty$.

We now consider the case when $p>\frac12$. Conditioning on the value of $S_N$, we have that
\begin{eqnarray*}
\lefteqn{\mathbf{P}\left((\eta^{N,iid}_n)_{n=1}^{M}=(x_n)_{n=1}^{M}\right)}\\
&=&\sum_{k>0}\mathbf{P}\left((\eta^{\frac12}_n)_{n=1}^{M}=(x_n)_{n=1}^{M}\:\vline\:S_N=k\right)
\mathbf{P}\left(S_N=k\:\vline\:S_N>0\right)\\
&=&\sum_{k>0}\frac{(N-M)!\left(\frac{N-k}{2}\right)!\left(\frac{N+k}{2}\right)!}{
N!\left(\frac{N-k}{2}-\sum_{n=1}^Mx_n\right)!\left(\frac{N+k}{2}-M+\sum_{n=1}^Mx_n\right)!}\mathbf{P}\left(S_N=k\:\vline\:S_N>0\right),
\end{eqnarray*}
where the summands should be interpreted as $0$ wherever the arguments of the terms involving factorials are not all non-negative integers. We next note that Cramer's theorem for an i.i.d.\ sequence (e.g.\ \cite[Theorem 2.2.3]{DZ}) yields that, for any $\varepsilon>0$,
\[\mathbf{P}\left(S_N>\varepsilon N\:\vline\:S_N>0\right)\rightarrow 0.\]
Moreover, straightforward calculations give that, uniformly over the relevant $k\in[0,\varepsilon N]$,
\[2^{-M}(1-\varepsilon)^M
\leq \frac{(N-M)!\left(\frac{N-k}{2}\right)!\left(\frac{N+k}{2}\right)!}{
N!\left(\frac{N-k}{2}-\sum_{n=1}^Mx_n\right)!\left(\frac{N+k}{2}-M+\sum_{n=1}^Mx_n\right)!}\leq
2^{-M}(1+\varepsilon)^M\left(\frac{1}{1-\frac{M-1}{N}}\right)^M.\]
It thus follows that
\[\mathbf{P}\left((\eta^{N,iid}_n)_{n=1}^{M}=(x_n)_{n=1}^{M}\right)\rightarrow 2^{-M}=\mathbf{P}\left((\eta^{\frac12}_n)_{n=1}^{M}=(x_n)_{n=1}^{M}\right),\]
as desired.
\end{proof}

\section{Continuous invariant measures}\label{contsec}

In \cite{CKSS}, a continuous state space version of the BBS was formulated to describe scaling limits of the discrete system. This was based on a two-sided version of Pitman's transformation for continuous functions, which had been studied previously in the probabilistic literature, particularly in the context of queuing (see, for example, \cite{OCY}). The main example given in \cite{CKSS} was the two-sided Brownian with drift (this is recalled in Subsection \ref{BMsec}), which had previously been shown to be invariant for Pitman's transformation in \cite{HW}. Here we further show that the zigzag process, which also appears in the queueing literature \cite{HMOC}, naturally arises as a limit of the Markov initial configuration, see Subsection \ref{zzsec}. Whilst it is possible to check that Brownian motion and the zigzag process are both invariant under Pitman's transformation directly, our approach is to deduce the latter results by establishing that the processes in question are scaling limits of discrete systems, and showing that the invariance under $T$ transfers to the limit. In addition to the examples already mentioned, we follow this line of argument for the periodic models described in Examples \ref{periidex} and \ref{permarex}, see Subsections \ref{perbmsec} and \ref{perzzsec}, respectively. We also discuss continuous versions of the bounded soliton examples of Subsection \ref{boundedsec} and Example \ref{perboundedex} in Subsection \ref{contbounded}.

Prior to introducing the specific models, let us summarise the scaling approach we will use. The following assumption describes the framework in which we are working.

\begin{assu}\label{sassu} It holds that $\eta^\varepsilon=(\eta^\varepsilon_n)_{n\in \mathbb{Z}}$, $\varepsilon>0$, is a collection of random configurations such that
\begin{equation}\label{invarassu}
T\eta^\varepsilon\buildrel{d}\over{=}\eta^\varepsilon
\end{equation}
for each $\varepsilon>0$. The corresponding path encodings $S^\varepsilon$, $\varepsilon>0$, satisfy
\begin{equation}\label{scaling}
\left(a_\varepsilon S^\varepsilon_{t/b_\varepsilon}\right)_{t\in\mathbb{R}}\buildrel{d}\over{\rightarrow} \left(S_t\right)_{t\in\mathbb{R}},
\end{equation}
in $C(\mathbb{R},\mathbb{R})$, where: $(a_\varepsilon)_{\varepsilon>0}$ and $(b_\varepsilon)_{\varepsilon>0}$ are deterministic sequences in $(0,\infty)$; $S^\varepsilon$ is extended to an element of $C(\mathbb{R},\mathbb{R})$ by linear interpolation; and $S$ is a random element of $C(\mathbb{R},\mathbb{R})$. Moreover, for any $t\in\mathbb{R}$, it holds that
\begin{equation}\label{cond1}
\lim_{s \to -\infty} \limsup_{\varepsilon\rightarrow 0} \mathbf{P}\left(M^{\varepsilon}_{s/b_{\varepsilon}} > S^{\varepsilon}_{t/b_\varepsilon}\right)=0,
\end{equation}
and
\begin{equation}\label{cond2}
\lim_{s \to -\infty} \mathbf{P}\left(M_{s} > S_{t}\right)=0,
\end{equation}
where $M^{\varepsilon}$ and $M$ are the past maximum processes associated with $S^{\varepsilon}$ and $S$, respectively.
\end{assu}

We note that the conditions at \eqref{cond1} and \eqref{cond2} ensure the simultaneous convergence of the rescaled past maximum processes with the convergence of path encodings given at \eqref{scaling}, and as a consequence we obtain the following result concerning the invariance under $T$ of the limiting path encoding.

\begin{prop}[cf.\ {\cite[Lemma 5.11]{CKSS}}]\label{invarlimit} If Assumption \ref{sassu} holds, then
\[TS\buildrel{d}\over{=}S.\]
\end{prop}

\subsection{Brownian motion with drift}\label{BMsec} Perhaps the simplest, and most fundamental, (non-trivial) example of a scaling limit for the path encoding of the box-ball system is seen in the high-density regime. Specifically, fix a constant $c>0$, and consider the configuration $\eta^\varepsilon$ generated by an i.i.d.\ sequence of Bernoulli random variables, with parameter
\begin{equation}\label{peps}
p_\varepsilon:=\frac{1-\varepsilon c}{2}.
\end{equation}
(We assume $\varepsilon<c^{-1}$ for the above to make sense.) By Corollary \ref{revcor}, we have that \eqref{invarassu} holds. Moreover, it is an elementary application of the classical invariance principle that \eqref{scaling} holds with $a_\varepsilon=\varepsilon$, $b_\varepsilon=\varepsilon^2$, and $S$ a two-sided Brownian motion with drift $c$, i.e.\
\[S_t=\left\{\begin{array}{ll}
               ct+S^{(1)}_t, & t\geq 0,\\
               ct+S^{(2)}_{-t}, & t<0,
             \end{array}\right.\]
where $S^{(1)}$ and $S^{(2)}$ are independent standard Brownian motions (starting from 0). Also, \eqref{cond1} and \eqref{cond2} were checked as \cite[Lemma 5.12]{CKSS}. Hence Assumption \ref{sassu} is satisfied in this setting, and we conclude from Proposition \ref{invarlimit} the following result.

\begin{prop} If $S$ is a two-sided Brownian motion with drift $c>0$, then $TS\buildrel{d}\over{=}S$.
\end{prop}

\begin{rem}\label{bmcar} In this case, the carrier $W=M-S$ is the stationary version of Brownian motion with drift $-c$, reflected at the origin. In particular, $W_0$ is exponentially distributed with parameter $2c$, so that $\mathbf{E}W_0=(2c)^{-1}$.
\end{rem}

\subsection{Zigzag process}\label{zzsec} It is not difficult to extend the result of the previous section to show that Brownian motion with drift can also be obtained from a more general class of Markov configurations in the high-density limit. In this section, however, we study a different scaling regime for the Markov configurations of Section \ref{markovsec}. Indeed, we will consider the case when the adjacent states are increasingly likely to be the same, and explain how we can see the so-called zigzag process (we take the name from \cite{DMOC}, though there the name was applied to the carrier process $M-S$; our version is also a generalisation of the so-called telegrapher's process \cite{Kac}) as a scaling limit.

Concerning the details, in this section we fix $\lambda_0,\lambda_1>0$, and suppose $\eta^\varepsilon$ is a two-sided stationary Markov chain on $\{0,1\}$ with transition matrix
\begin{equation}\label{Peps}
P_\varepsilon=\left(
      \begin{array}{cc}
         1- \varepsilon \lambda_0 & \varepsilon \lambda_0 \\
         \varepsilon \lambda_1 & 1- \varepsilon \lambda_1 \\
      \end{array}
    \right).
\end{equation}
(We assume $\varepsilon$ is small enough so that the entries of this matrix are strictly positive.) We note that the invariant measure for $\eta^{\varepsilon}$ is independent of $\varepsilon$, being given by
\[\mathbf{P}\left(\eta^\varepsilon_0=1\right)=\frac{\lambda_0}{\lambda_0+\lambda_1},\]
and so to ensure the associated path encoding has distribution supported on $\mathcal{S}^{lin}$, we thus need to assume $\lambda_0<\lambda_1$, as we will do henceforth. From Corollary \ref{markovcor}, we then have that \eqref{invarassu} holds in this setting.

By definition, the numbers of spatial locations for which $\eta^{\varepsilon}$ takes the value $0$ or $1$ before a change are given by geometric random variables with parameters $\varepsilon \lambda_0$ or $\varepsilon \lambda_1$, respectively. Noting that, when multiplied by $\varepsilon$, the latter random variables converge to exponential, parameter $\lambda_0$ or $\lambda_1$, random variables, it is an elementary exercise to check that
\begin{equation}\label{etaconv}
\left(\eta^\varepsilon_{\lfloor t/\varepsilon \rfloor}\right)_{t\in\mathbb{R}}\buildrel{d}\over{\rightarrow} \left(\eta_t\right)_{t\in\mathbb{R}}
\end{equation}
in $D(\mathbb{R},\{0,1\})$, where the limiting process is the two-sided, stationary continuous-time Markov chain on $\{0,1\}$ that jumps from $0$ to $1$ with rate $\lambda_0$, and from $1$ to $0$ with rate $\lambda_1$. As a consequence, we find that \eqref{scaling} holds with $a_\varepsilon=b_{\varepsilon}=\varepsilon$, and the limiting process being given by $S=(S_t)_{t\in\mathbb{R}}$, where
\begin{equation}\label{sdef}
S_t:=\int_0^t(1-2\eta_s)ds;
\end{equation}
this is the zigzag process. Since $\lambda_0<\lambda_1$, it is an elementary to exercise to check that $t^{-1}S_t\rightarrow \frac{\lambda_1-\lambda_0}{\lambda_1+\lambda_0}>0$ as $|t|\rightarrow \infty$, $\mathbf{P}$-a.s., from which \eqref{cond2} readily follows. The remaining condition we need to apply Proposition \ref{invarlimit} is given by the following lemma.

\begin{lem} If $\eta^{\varepsilon}$ are the random configurations described above with $\lambda_0<\lambda_1$, then \eqref{cond1} holds with $b_\varepsilon=\varepsilon$.
\end{lem}
\begin{proof} Applying the Markov property for $\eta^\varepsilon$ it will suffice to show that
\[\lim_{t \to -\infty} \limsup_{\varepsilon\rightarrow 0} \mathbf{P}\left(M^{\varepsilon}_{t/{\varepsilon}} > 0\:\vline\:\eta^{\varepsilon}_0=i\right)=0,\qquad i=0,1.\]
To this end, observe that, for any $x\geq 0$,
\[\mathbf{P}\left(M^{\varepsilon}_{t/{\varepsilon}} > 0\:\vline\:\eta^{\varepsilon}_0=i\right)\leq
\mathbf{P}\left(\varepsilon S^{\varepsilon}_{t/{\varepsilon}} > -x\:\vline\:\eta^{\varepsilon}_0=i\right)+\sup_{j\in\{0,1\}}\mathbf{P}\left(\varepsilon M^{\varepsilon}_{0} > x\:\vline\:\eta^{\varepsilon}_0=j\right).\]
The first term on the right-hand side here is readily checked to converge to
$\mathbf{P}(S_{t} > -x\:\vline\:\eta_0=i)$ as $\varepsilon\rightarrow 0$, and this limit converges to 0 as $t\rightarrow-\infty$. As for the second term, from \eqref{w0pm} we have that
\begin{eqnarray*}
\sup_{j\in\{0,1\}}\mathbf{P}\left(\varepsilon M^{\varepsilon}_{0} > x\:\vline\:\eta^{\varepsilon}_0=j\right)&\leq & \frac{\lambda_0+\lambda_1}{\lambda_0}\mathbf{P}\left(\varepsilon M^{\varepsilon}_{0} > x\right)\\
&\leq & C\varepsilon\sum_{m>x/\varepsilon}\left(\frac{1-\lambda_1\varepsilon}{1-\lambda_0\varepsilon}\right)^{m}\\
&\leq & C\varepsilon\sum_{m>x/\varepsilon}e^{-(\lambda_1-\lambda_0)\varepsilon m}\\
&\leq & Ce^{-(\lambda_1-\lambda_0)x},
\end{eqnarray*}
where $C$ is a constant not depending on $\varepsilon$ that might vary from line to line. This expression can be taken arbitrarily small by choosing $x$ large, and so the proof is complete.
\end{proof}

\begin{prop} If $S$ is the zigzag process with parameters $\lambda_0<\lambda_1$, then $TS\buildrel{d}\over{=}S$.
\end{prop}

\begin{rem} In this case, the carrier $W=M-S$ is a stationary, non-Markov process. It is possible to compute its marginal distribution by taking the appropriate scaling limit of the distribution given in Lemma \ref{w0pm}, yielding
\[W_0\sim\frac{\lambda_1-\lambda_0}{\lambda_1+\lambda_0}\delta_0+\frac{2\lambda_0}{\lambda_1+\lambda_0}\mathrm{Exp}(\lambda_1-\lambda_0),\]
where $\delta_0$ is the probability measure placing all its mass at $0$, and $\mathrm{Exp}(\lambda_1-\lambda_0)$ is the law of an exponential random variable with parameter $\lambda_1-\lambda_0$. In particular, $\mathbf{E}W_0=2\lambda_0(\lambda_1^2-\lambda_0^2)^{-1}$.
\end{rem}

\subsection{Periodic Brownian motion}\label{perbmsec} In this subsection, we describe the periodic version of the scaling argument of Subsection \ref{BMsec}. Let $\eta^{\varepsilon}$ be again an i.i.d.\ sequence of Bernoulli random variables, with parameter $p_\varepsilon$ as given by \eqref{peps}, for some constant $c\in \mathbb{R}$. (Note that we no longer need to assume $c>0$.) Moreover, for $L>0$, set $N_\varepsilon:=\lfloor L/\varepsilon^2\rfloor$, and let $(\eta^{\varepsilon,L}_n)_{n=1}^{N_\varepsilon}$ be a random sequence with law given by that of $(\eta^{\varepsilon}_n)_{n=1}^{N_\varepsilon}$ conditioned on $S^{\varepsilon}_{N_{\varepsilon}}>0$. Extend $\eta^{\varepsilon,L}$ to $(\eta^{\varepsilon,L}_n)_{n\in\mathbb{Z}}$ by cyclic repetition. From Proposition \ref{gibbsinv}, we then have that $T\eta^{\varepsilon,L}\buildrel{d}\over{=}\eta^{\varepsilon,L}$, and so \eqref{invarassu} holds for these random configurations. Moreover, it is straightforward to check that \eqref{scaling} holds, in the sense that the associated path encodings satisfy
\[\left(\varepsilon S^{\varepsilon,L}_{t/\varepsilon^2}\right)_{t\in\mathbb{R}}\buildrel{d}\over{\rightarrow}
\left( S^{L}_{t}\right)_{t\in\mathbb{R}},\]
where $(S^L_t)_{t\in[0,L]}$ has the distribution of the initial segment of two-sided Brownian motion with drift $c$, $(S_t)_{t\in[0,L]}$, conditioned on $S_L>0$, and this definition is extended by cyclic repetition to give a process on $\mathbb{R}$. With the latter definition, it is obvious that $t^{-1}S^{L}_t\rightarrow L^{-1}S_L>0$ as $|t|\rightarrow \infty$, $\mathbf{P}$-a.s., and so \eqref{cond2} holds. As for \eqref{cond1}, we simply note
\begin{eqnarray*}
\lim_{s \to -\infty} \limsup_{\varepsilon\rightarrow 0} \mathbf{P}\left(M^{\varepsilon,L}_{s/\varepsilon^2} > S^{\varepsilon,L}_{t/\varepsilon^2}\right)&=&
\lim_{s \to -\infty} \limsup_{\varepsilon\rightarrow 0} \mathbf{P}\left(M^{\varepsilon,L}_{s/\varepsilon^2} > 0\right)\\
&\leq &\lim_{s \to -\infty} \limsup_{\varepsilon\rightarrow 0} \mathbf{P}\left(\sup_{u\in[0,L]}\varepsilon S^{\varepsilon,L}_{u/\varepsilon^2}+\left\lfloor\frac{s}{L}\right\rfloor \varepsilon S^{\varepsilon,L}_{L/\varepsilon^2} > 0\right)\\
&=&\lim_{s \to -\infty} \mathbf{P}\left(\sup_{u\in[0,L]}S_u^{L}+\left\lfloor\frac{s}{L}\right\rfloor  S^{L}_{L} > 0\right)\\
&=&0,
\end{eqnarray*}
where $M^{\varepsilon,L}$ is the past maximum process associated with $S^{\varepsilon,L}$. Hence Assumption \ref{invarassu} holds, and we obtain the following.

\begin{prop} Fix $L>0$. If $S^L$ is the periodic extension of $(S_t)_{t\in[0,L]}$ conditioned on $S_L>0$, where $S$ is a two-sided Brownian motion with drift $c\in\mathbb{R}$, then $TS^L\buildrel{d}\over{=}S^L$.
\end{prop}

\subsection{Periodic zigzag process}\label{perzzsec} The periodic analogue of Subsection \ref{zzsec} is checked similarly to the previous subsection. For $L>0$, set $N_\varepsilon:=\lfloor L/\varepsilon\rfloor$, and let $(\eta^{\varepsilon,L}_n)_{n=1}^{N_\varepsilon}$ be a random sequence with law given by \eqref{cyclicmarkov}, where $N=N_\varepsilon$, and $(\eta_n)_{n=1}^{N_\varepsilon}$ is given by the two-sided stationary Markov chain with transition matrix $P_\varepsilon$ from \eqref{Peps} for some $\lambda_0,\lambda_1>0$. Extending $\eta^{\varepsilon,L}$ to $(\eta^{\varepsilon,L}_n)_{n\in\mathbb{Z}}$ by cyclic repetition, we then have from Proposition \ref{gibbsinv} that $T\eta^{\varepsilon,L}\buildrel{d}\over{=}\eta^{\varepsilon,L}$, and so \eqref{invarassu} holds for these random configurations. Moreover, it is not difficult to deduce from \eqref{etaconv} that
\[\left(\eta^{\varepsilon,L}_{\lfloor t/\varepsilon \rfloor}\right)_{t\in\mathbb{R}}\buildrel{d}\over{\rightarrow} \left(\eta^L_t\right)_{t\in\mathbb{R}}\]
in $D(\mathbb{R},\{0,1\})$, with the law of the $L$-periodic process $\eta^L$ being characterised by
\begin{equation}\label{etal}
\mathbf{E}\left(F\left((\eta^L_t)_{t\in[0,L]}\right)\right)=\frac{\mathbf{E}\left(\nu(\eta_0)^{-1}F\left((\eta_t)_{t\in[0,L]}\right)\:\vline\:\eta_L=\eta_0,\:S_L>0\right)}{\mathbf{E}\left(\nu(\eta_0)^{-1}\:\vline\:\eta_L=\eta_0,\:S_L>0\right)},
\end{equation}
where $1-\nu(0)=\frac{\lambda_0}{\lambda_0+\lambda_1}=\nu(1)$, and $\eta$ is the two-sided, stationary continuous-time Markov chain that appears as a limit in \eqref{etaconv}. It follows that the associated path encodings satisfy
\[\left(\varepsilon S^{\varepsilon,L}_{t/\varepsilon }\right)_{t\in\mathbb{R}}\buildrel{d}\over{\rightarrow} \left(S^L_t\right)_{t\in\mathbb{R}},\]
yielding \eqref{scaling} in this case; the limit process can be seen as a periodic version of the zigzag process with stationary increments. By applying identical arguments to those of the previous subsection, we are also able to confirm \eqref{cond1} and \eqref{cond2} both hold with the appropriate scaling, and we subsequently obtain the following.

\begin{prop} Fix $L>0$. If $S^L$ is the path encoding of the $\eta^L$, as given by \eqref{etal}, for some $\lambda_0,\lambda_1>0$, then $TS^L\buildrel{d}\over{=}S^L$.
\end{prop}

\subsection{Brownian motion conditioned to stay close to its past maximum}\label{contbounded}

In this section, we consider the transfer of the bounded soliton examples of Subsection \ref{boundedsec} and Example \ref{perboundedex} to the continuous setting, starting with the periodic case. Let $L>0$, and $S^L$ be the $L$-periodic Brownian motion with drift $c>0$ of Subsection \ref{perbmsec}. If $W^L=M^L-S^L$ is the associated carrier process and $K>0$, we define $S^{L,K}$ to have law equal to that of $S^L$ conditioned on $\sup_{t\in\mathbb{R}}W^L_t\leq K$. (Note the latter event has strictly positive probability.) We then have the following.

\begin{prop} Fix $L,K>0$. If $S^{L,K}$ is the $L$-periodic Brownian motion with drift $c>0$ conditioned to stay within $K$ of its past maximum (i.e.\ the process described above), then $TS^{L,K}\buildrel{d}\over{=}S^{L,K}$.
\end{prop}
\begin{proof} Note that $\sup_{t\in\mathbb{R}}W^L_t\leq K$ can alternatively be expressed as
\begin{equation}\label{alk}
\left\{\max_{0\leq t\leq L}\left\{\max_{0\leq s\leq t}(S^L_s-S^L_t),\:\max_{t\leq s\leq L}(S^L_s-S^L_L-S^L_t)\right\}\leq K\right\}.
\end{equation}
Hence, applying the definitions of $S^{L,K}$ and $S^L$, we find that
\begin{equation}\label{123}
\mathbf{E}\left(F\left((S^{L,K}_t)_{t\in[0,L]}\right)\right)=\mathbf{E}\left(F\left((S_t)_{t\in[0,L]}\right)\:\vline\:\mathcal{A}_{L,K},\:S_L>0\right),
\end{equation}
where $\mathcal{A}_{L,K}$ is defined similarly to \eqref{alk}, but with $S^L$ replaced by $S$. This characterisation of the law of $S^{L,K}$ allows us to show that it can be arrived at as the scaling limit of a sequence of discrete models. Indeed, let $S^{\varepsilon,L,K}$ be the periodic bounded soliton configuration of Example \ref{perboundedex} with $(p,N,K)$ being given by $(\frac{1-\varepsilon c}{2},\lfloor L/\varepsilon^2\rfloor,K/\varepsilon)$. From \eqref{bexp0}, we then have that
\begin{equation}\label{456}
\mathbf{E}\left(F\left((S^{\varepsilon,L,K}_n)_{n=0}^{\lfloor L/\varepsilon^2\rfloor}\right)\right)=\mathbf{E}\left(F\left((S^{\varepsilon}_n)_{n=0}^{\lfloor L/\varepsilon^2\rfloor}\right)\:\vline\:\mathcal{A}_{\varepsilon,L,K},\:S^\varepsilon_{\lfloor L/\varepsilon^2\rfloor}>0\right),
\end{equation}
where $S^{\varepsilon}$ is the path encoding of the i.i.d.\ configuration with density $\frac{1-\varepsilon c}{2}$, and
\[\mathcal{A}_{\varepsilon,L,K}=\left\{\max_{0\leq n\leq {\lfloor L/\varepsilon^2\rfloor}}\left\{\max_{0\leq m\leq n}(S^\varepsilon_m-S^\varepsilon_n),\:\max_{n\leq m\leq {\lfloor L/\varepsilon^2\rfloor}}(S^\varepsilon_m-S^\varepsilon_{\lfloor L/\varepsilon^2\rfloor}-S^\varepsilon_n)\right\}\leq K/\varepsilon\right\}.\]
Since $(\varepsilon S^{\varepsilon}_{t/\varepsilon^2})_{t\in\mathbb{R}}\buildrel{d}\over{\rightarrow}S$, it is an elementary exercise to deduce from \eqref{123} and \eqref{456} that
\[\left(\varepsilon S^{\varepsilon,L,K}_{t/\varepsilon^2}\right)_{t\in\mathbb{R}}\buildrel{d}\over{\rightarrow}\left(S^{L,K}_t\right)_{t\in\mathbb{R}},\]
i.e.\ \eqref{scaling} holds with $a_\varepsilon=\varepsilon$, $b_\varepsilon=\varepsilon^2$. We also have \eqref{invarassu} from Corollary \ref{gibbsinv}, and \eqref{cond1} and \eqref{cond2} can be checked as in Subsection \ref{perbmsec}. Hence Assumption \ref{sassu} holds, and Proposition \ref{invarlimit} yields the result.
\end{proof}

The non-periodic version of the previous result is more of a challenge, and we do not prove it here. Rather we describe a potential proof strategy. Firstly, recall from Remark \ref{bmcar} that the carrier process $W=M-S$ associated with Brownian motion with drift $c>0$ is the stationary version of Brownian motion with drift $-c$, reflected at the origin. By applying \cite[Section 4]{GT}, it is possible to define a stationary Markov process ${W}^{K}$ that can be interpreted as $W$ conditioned on $\sup_{t\in\mathbb{R}}W_t\leq K$ (cf.\ the discussion for reflecting Brownian motion without drift in \cite[Section 7]{GT2}). Letting $L^K$ be the local time at 0 of this process, with boundary condition $L^K_0=0$, then, by analogy with the unconditioned case, set $S^{K}=L^K-W^K+W^K_0$. We expect that this process, which one might interpret as Brownian motion with drift $c>0$ conditioned to stay within $K$ of its past maximum, can alternatively be obtained as a scaling limit of the path encodings of the random configurations described in Subsection \ref{boundedsec}, and make the following conjecture.

\begin{conj} Fix $K>0$. If $S^{K}$ is the Brownian motion with drift $c>0$ conditioned to stay within $K$ of its past maximum (in the sense described above), then $TS^{K}\buildrel{d}\over{=}S^{K}$.
\end{conj}

\section{Palm measures for the zigzag process and the ultra-discrete Toda lattice}\label{palmsec}

In this section we relate the dynamics of the zigzag process under Pitman's transformation to the dynamics of the ultra-discrete Toda lattice, and use this connection to derive natural invariant measures for the latter. The state of the ultra-discrete Toda lattice is described by a vector $((Q_j)_{j=1}^J,({E}_j)_{j=1}^{J-1})  \in (0,\infty)^{2J-1}$ for some $J\in\mathbb{N}$, and its one-step time evolution by the equation
\begin{align}
(\mathcal{T}Q)_j &:=\min \left\{ \sum_{l=1}^j Q_l-\sum_{l=1}^{j-1}(\mathcal{T}Q)_l,{E}_j \right\},\label{udtodaeq}\\
(\mathcal{T}{E})_j &:=Q_{j+1}+{E}_j-(\mathcal{T}Q)_j,\nonumber
\end{align}
where for the purposes of these equations we suppose ${E}_{J} = \infty$. Similarly to the path encoding of the BBS, we can associate a path $S\in C(\mathbb{R},\mathbb{R})$ to the state of the ultra-discrete Toda lattice $((Q_j)_{j=1}^J,({E}_j)_{j=1}^{J-1})$ by setting $S_t=t$ for $t<0$, and for $t\geq 0$, concatenating path segments of gradient $-1,1,-1,1,\dots,-1,1,-1,1$, of lengths $Q_1,{E}_1,Q_2,{E}_2,\dots,Q_{J-1},{E}_{J-1},Q_{J},{E}_J=\infty$, i.e.\
\begin{equation}\label{abpath}
S_{t}=\begin{cases}
t&\mbox{for }t<0,\\
-t+2\sum_{l=1}^j{E}_l,&\mbox{for }\sum_{l=1}^j Q_l+\sum_{l=1}^j{E}_l\leq t\leq\sum_{l=1}^{j+1}Q_l+\sum_{l=1}^j{E}_l,\\
t-2\sum_{l=1}^{j+1}Q_l,&\mbox{for }\sum_{l=1}^{j+1}Q_l+\sum_{l=1}^j{E}_l\leq t\leq\sum_{l=1}^{j+1}Q_l+\sum_{l=1}^{j+1}{E}_l,
\end{cases}
\end{equation}
where $j=0,\dots,J-1$ (interpreting sums of the form $\sum_{l=1}^0$ as zero), and we again suppose ${E}_{J} = \infty$. As is confirmed by the next proposition we present, the dynamics of the ultra-discrete Toda lattice given by \eqref{udtodaeq} are described by Pitman's transformation applied to this path encoding. However, in this case, it is convenient to shift the path after applying $T$ so that $0$ is still a local maximum. In particular, for $t\in\mathbb{R}$, we define $\theta^tS$ by setting
\begin{equation}\label{thetatdef}
(\theta^tS)_s=S_{t+s}-S_t,\qquad\forall s\in\mathbb{R},
\end{equation}
let
\[\tau(S):=\inf\{t\geq 0:\:t\in \mathrm{LM}(S)\},\]
where $\mathrm{LM}(S)$ is the set of local maxima of $S$ (for the elements of $C(\mathbb{R},\mathbb{R})$ that are considered in this section, $\tau(S)$ is always well-defined and finite), and define
\[\theta^\tau(S):=\theta^{\tau(S)}(S).\]
We then introduce an operator $\mathcal{T}$ on the path encoding by the composition of $T$ and $\theta^\tau$, that is
\begin{equation}\label{tstarpath}
\mathcal{T}S:=\theta^\tau(TS).
\end{equation}
The motivation for this definition is the following. (See Figure \ref{todafig} for a graphical representation of the result.)

\begin{prop}[See {\cite[Theorem 1.1]{CST}}] Fix $J\in\mathbb{N}$. Let $((Q_j)_{j=1}^J,({E}_j)_{j=1}^{J-1})  \in (0,\infty)^{2J-1}$, and $S$ be its path encoding, defined as at \eqref{abpath}. It is then the case that the transformed configuration $(((\mathcal{T}Q)_j)_{j=1}^J,((\mathcal{T}{E})_j)_{j=1}^{J-1})$, defined as at \eqref{udtodaeq}, has path encoding given by $\mathcal{T}S$, defined as at \eqref{tstarpath}.
\end{prop}

\begin{figure}[t]
\begin{flushleft}
\hspace{60pt}\includegraphics[width=0.5\textwidth]{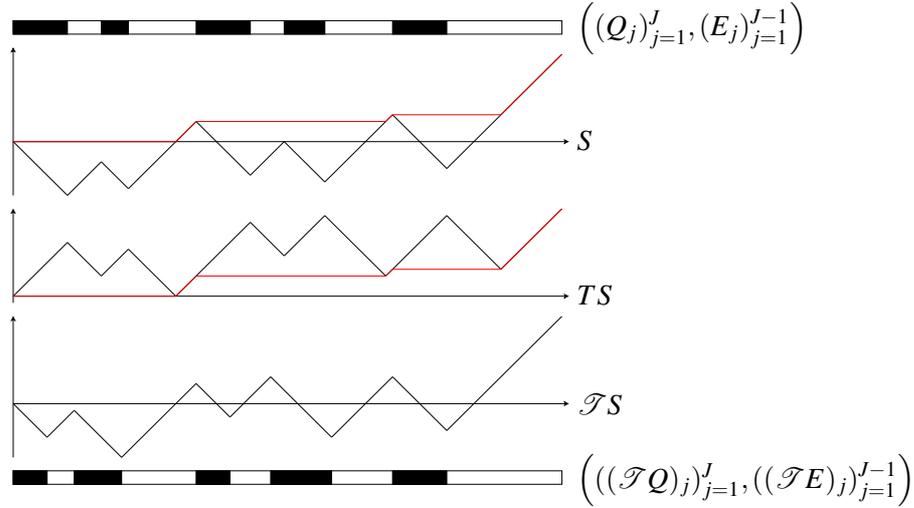}
\rput[tl](0,6.4){$\left((Q_j)_{j=1}^J,({E}_j)_{j=1}^{J-1}\right)$}
\rput[tl](0,4.7){$S$}
\rput[tl](0,2.65){$TS$}
\rput[tl](0,1.2){$\mathcal{T}S$}
\rput[tl](0,0.4){$\left(((\mathcal{T}Q)_j)_{j=1}^J,((\mathcal{T}{E})_j)_{j=1}^{J-1}\right)$}
\end{flushleft}
\caption{Graphical representation of the dynamics of the ultra-discrete Toda lattice in terms of the associated path encodings.  NB. The red line in the graphs for $S$ and $TS$ shows the path of $M$.}\label{todafig}
\end{figure}

Just as for the BBS, the ultra-discrete Toda lattice evolves in a solitonic way. Eventually the configuration orders itself so that $Q_J\geq Q_{J-1}\geq\dots\geq Q_1$, and these quantities -- which can be thought of as representing intervals where particles are present -- remain constant, whilst the $({E}_j)_{j=1}^{J-1}$ -- which can be thought of as representing the gaps between blocks of particles -- grow linearly (see \cite[equations (20), (21)]{NTS}, though note the labelling convention is reversed in the latter article). Thus to see a stationary measure one might consider, as we did for the BBS, a two-sided infinite configuration $((Q_j)_{j\in\mathbb{Z}},({E}_j)_{j\in\mathbb{Z}})$. Under suitable conditions regarding the asymptotic behaviour of these sequences, one might then encode these via piecewise linear paths with intervals of gradient $-1$ or $1$ as at \eqref{abpath} -- extending the definition to the negative axis in the obvious way, and then defining the dynamics via \eqref{tstarpath}. This is our approach in the next part of our discussion.  Although $\mathcal{T}$ is a more complicated operator than $T$, we are still able to identify an invariant measure for it by considering the Palm measure of the zigzag process under which $0$ is always a local maximum. As we show in Corollary \ref{abcor}, reading off the lengths of the intervals of constant gradient, from the latter conclusion we obtain a natural invariant measure for the ultra-discrete Toda lattice. Specifically, the invariant configuration we present has that both $(Q_j)_{j\in\mathbb{Z}}$ and $({E}_j)_{j\in\mathbb{Z}}$ are i.i.d.\ sequences of exponential random variables (independent of each other).

The result described in the previous paragraph for the Palm measure of the zigzag process, and the corollary for the ultra-discrete Toda lattice, will be proved in Subsection \ref{zzpalm}. Towards this end, in Subsection \ref{marpalm}, we first establish the BBS analogue of the results for the Markov configuration of Subsection \ref{markovsec}. Finally, in Subsection \ref{perpalm}, we establish periodic versions of the results.

\subsection{Invariance of a Palm measure for the Markov configuration}\label{marpalm} In this subsection, we suppose $\eta$ is the Markov configuration of Subsection \ref{markovsec} with $p_0,p_1\in(0,1)$ and $p_0+p_1<1$. The associated Palm measure we will consider is defined to be the law of the random configuration $\eta^*$, as characterised by
\begin{equation}\label{etastar}
\mathbf{E}\left(f(\eta^*)\right)=\mathbf{E}\left(f(\eta)\:\vline\:\eta_0=0,\:\eta_1=1\right)
\end{equation}
for any bounded functions $f:\{0,1\}^\mathbb{Z}\rightarrow\mathbb{R}$. Equivalently, we can express this in terms of the associated path encodings as
\[\mathbf{E}\left(f(S^*)\right)=\mathbf{E}\left(f(S)\:\vline\:0\in \mathrm{LM}(S)\right).\]
The main result of the subsection is the following, which establishes invariance of $S^*$ under $\mathcal{T}$. The proof is an adaptation of \cite[Lemma 4.5]{Ferrari}, cf.\ the classical arguments of \cite{Har,PS}.

\begin{prop}\label{propperpalmdisc} If $S^*$ is the path encoding of the two-sided stationary Markov chain described in Subsection \ref{markovsec} with $p_0,p_1\in (0,1)$ satisfying $p_0+p_1<1$ conditioned to have a local maximum at $0$, then $\mathcal{T}S^*\buildrel{d}\over{=}S^*$.
\end{prop}
\begin{proof} By definition, writing $c=\mathbf{P}(0\in \mathrm{LM}(S))^{-1}$, we have that
\begin{eqnarray}
\mathbf{E}\left(f(\mathcal{T}S^*)\right)&=&c\mathbf{E}\left(f(\theta^\tau(TS))\mathbf{1}_{\{0\in \mathrm{LM}(S)\}}\right)\nonumber\\
&=&c\sum_{n>0}\mathbf{E}\left(f(\theta^n(TS))\mathbf{1}_{\{0\in \mathrm{LM}(S),\:\tau(TS)=n\}}\right),\label{asf}
\end{eqnarray}
where we note that $\tau(TS)>0$ on the event $0\in \mathrm{LM}(S)$, and $\theta^n$ is defined as at \eqref{thetatdef}. Now, it is an elementary exercise to check that, on $0\in \mathrm{LM}(S)$, the event $\tau(TS)=n$ is equivalent to $\bar{\tau}(S)=n$, where $\bar{\tau}(S):=\inf\{n\geq 0:\:n\in \mathrm{LI}(S)\}$, and $\mathrm{LI}(S)$ is the set of local minima of $S$. Hence we obtain from \eqref{asf} that
\begin{eqnarray*}
\mathbf{E}\left(f(\mathcal{T}S^*)\right)&=&c\sum_{n>0}\mathbf{E}\left(f(\theta^n(TS))\mathbf{1}_{\{0\in \mathrm{LM}(S),\:\bar{\tau}(S)=n\}}\right)\\
&=&c\sum_{n>0}\mathbf{E}\left(f(T\theta^nS)\mathbf{1}_{\{\tau_-(\theta^nS)=-n,\:0\in LI(\theta^nS)\}}\right),
\end{eqnarray*}
where we define $\tau_-(S):=\sup\{n\leq 0:\:n\in LM(S)\}$. Applying the spatial stationarity of $\eta$, it follows that
\begin{eqnarray*}
\mathbf{E}\left(f(\mathcal{T}S^*)\right)&=&c\sum_{n>0}\mathbf{E}\left(f(TS)\mathbf{1}_{\{\tau_-(S)=-n,\:0\in LI(S)\}}\right)\\
&=&c\mathbf{E}\left(f(TS)\mathbf{1}_{\{0\in LI(S)\}}\right).
\end{eqnarray*}
Finally, we note that $0\in LI(S)$ if and only if $0\in LM(TS)$, and so
\[\mathbf{E}\left(f(\mathcal{T}S^*)\right)=c\mathbf{E}\left(f(TS)\mathbf{1}_{\{0\in LM(TS)\}}\right)
=c\mathbf{E}\left(f(S)\mathbf{1}_{\{0\in LM(S)\}}\right)
=\mathbf{E}\left(f(S^*)\right),\]
where the second equality follows from the invariance of $S$ under $T$ (i.e.\ Corollary \ref{markovcor}).
\end{proof}

\subsection{Invariance of a Palm measure for the zigzag process}\label{zzpalm} Via a scaling limit, the result of the previous subsection readily transfers to the zigzag process. In particular, given $\lambda_0<\lambda_1$, now let $\eta^*=(\eta^*_t)_{t\in\mathbb{R}}$ be a continuous time stochastic process taking values on $\{0,1\}$ such that: $(\eta^*_t)_{t\geq0}$ is a continuous time Markov chain that jumps from 0 to 1 with rate $\lambda_0$, and from 1 to 0 with rate $\lambda_1$, started from $\eta^*_0=1$; $(\eta^*_{-t})_{t\geq0}$ is a continuous time Markov chain with the jumps from 0 to 1 with rate $\lambda_0$, and from 1 to 0 with rate $\lambda_1$, started from $\eta^*_0=0$; and the two processes are independent. (NB. To make the process $\eta^*$ right-continuous, we ultimately set $\eta^*_0=1$, and also take the right-limits at all the jump times.) Our Palm measure for the zigzag process is then the law of $S^*=(S^*_t)_{t\in\mathbb{R}}$, where
\[S^*_t:=\int_0^t(1-2\eta^*_s)ds,\]
which can be viewed as the zigzag process $S$ of Subsection \ref{zzsec} conditioned on $0\in\mathrm{LM}(S)$, though in this case we note the conditioning is non-trivial since the event $0\in\mathrm{LM}(S)$ has zero probability. For the process $S^*$, we have the following result.

\begin{prop}\label{uuu} If $S^*$ is the zigzag process with rates $0<\lambda_0<\lambda_1$ conditioned to have a local maximum at 0 (in the sense described above), then $\mathcal{T}S^*\buildrel{d}\over{=}S^*$.
\end{prop}
\begin{proof} Let $\eta^{*,\varepsilon}$ be the process defined at \eqref{etastar} for parameters $p_0=\varepsilon\lambda_0$ and $p_1=1-\varepsilon\lambda_1$. Then, similarly to \eqref{etaconv}, it is straightforward to check that
\[\left(\eta^{*,\varepsilon}_{\lfloor t/\varepsilon \rfloor}\right)_{t\in\mathbb{R}}\buildrel{d}\over{\rightarrow} \left(\eta^*_t\right)_{t\in\mathbb{R}},\]
and hence the associated path encodings satisfy
\[\left(\varepsilon S^{*,\varepsilon}_{t/\varepsilon}\right)_{t\in\mathbb{R}}\buildrel{d}\over{\rightarrow} \left(S^*_t\right)_{t\in\mathbb{R}}.\]
Moreover, the conditions \eqref{cond1} and \eqref{cond2} are readily checked in this setting. From these facts, together with the readily-checked observation that $\varepsilon\tau(TS^{*,\varepsilon})\buildrel{d}\over{\rightarrow} \tau(TS^*)$ (simultaneously with the convergence of path encodings), the result follows by a simple adaptation of the argument of Proposition \ref{invarlimit}.
\end{proof}

Since the lengths of the intervals upon which $S^{*}$ is decreasing are i.i.d.\ parameter $\lambda_1$ exponential random variables, and the lengths of the intervals upon which it is increasing are i.i.d.\ parameter $\lambda_0$ exponential random variables (and the two collections are independent), we immediately deduce the following conclusion from the previous result (and the description of the ultra-discrete Toda lattice given at the start of the section).

\begin{cor}\label{abcor} Let $(Q_j)_{j\in\mathbb{Z}}$ be an i.i.d.\ sequence of parameter $\lambda_1$ exponential random variables, and  $({E}_j)_{j\in\mathbb{Z}}$ be an i.i.d.\ sequence of parameter $\lambda_0$ exponential random variables. Suppose further  $(Q_j)_{j\in\mathbb{Z}}$ and $({E}_j)_{j\in\mathbb{Z}}$ are independent. If $0<\lambda_0<\lambda_1$, then the distribution of $((Q_j)_{j\in\mathbb{Z}},({E}_j)_{j\in\mathbb{Z}})$ is invariant under the dynamics of the ultra-discrete Toda lattice.
\end{cor}

\subsection{Palm measures in the periodic case}\label{perpalm} The arguments of the previous two subsections are readily adapted to the periodic case. Since few changes are needed, we only present a sketch, beginning with the discrete case. For $N\in\mathbb{N}$, let $\eta^{*N}=(\eta^{*N}_n)_{n\in\mathbb{Z}}$ be the random configuration with law characterised by
\begin{equation}\label{ggg}
\mathbf{E}\left(f(\eta^{*N})\right)=\mathbf{E}\left(f(\eta^N)\:\vline\:\eta^N_0=0,\:\eta^N_1=1\right),
\end{equation}
where $\eta^N$ is the periodic Markov configuration of Example \ref{permarex}, cf.\ \eqref{etastar}. Note that an alternative characterisation of the law of $\eta^{*N}$ is given by
\begin{equation}\label{altchar}
\mathbf{E}\left(f((\eta^{*N}_n)_{n=1}^N)\right)=\mathbf{E}\left(f((\eta_n)_{n=1}^N)\:\vline\:\eta_1=1,\:\eta_N=0,\:S_N>0\right),
\end{equation}
where $\eta$ is the Markov configuration of Subsection \ref{markovsec}. We are then able to check the following result.

\begin{prop} Let $p_0,p_1\in(0,1)$. If $S^{*N}$ is the path encoding of $\eta^{*N}$, then $\mathcal{T}S^{*N}\buildrel{d}\over{=}S^{*N}$.
\end{prop}
\begin{proof} This is identical to the proof of Proposition \ref{propperpalmdisc}. In particular, in view of the Palm description of the law of $\eta^{*N}$ at \eqref{ggg}, it suffices to note that $\eta^N$ is spatially stationary (Lemma \ref{etalem}) and invariant under $T$ (Corollary \ref{gibbsinv}), that $\mathbf{P}(0\in LM(S^N))>0$, and that the terms involving $\tau$, $\bar{\tau}$ and $\tau_-$ are almost-surely finite.
\end{proof}

For the continuous version of this result, first let $\eta^{*L}=(\eta^{*L}_t)_{t\in\mathbb{R}}$ be the $L$-periodic process whose law is characterised by
\begin{equation}\label{altchar2}
\mathbf{E}\left(F\left((\eta^{*L}_t)_{t\in[0,L]}\right)\right)=\mathbf{E}\left(F\left((\eta_t)_{t\in[0,L]}\right)\:\vline\:\eta_0=1,\:\eta_L=0,\:S_L>0\right),
\end{equation}
where $\eta$ is the two-sided stationary continuous time Markov chain of Subsection \ref{zzsec}. (NB. Of course, this definition is problematic in terms of defining $\eta_t^{*L}$ for $t\in L\mathbb{Z}$; we resolve the issue by assuming $\eta^{*L}$ is right-continuous.) If $S^{*L}$ is the corresponding path encoding, defined similarly to \eqref{sdef}, then we have the following result.

\begin{prop}\label{perpropsl} Let $\lambda_0,\lambda_1>0$. If $S^{*L}$ is the path encoding of $\eta^{*L}$, then $\mathcal{T}S^{*L}\buildrel{d}\over{=}S^{*L}$.
\end{prop}
\begin{proof} Similarly to the proof of Proposition \ref{uuu}, we use a scaling argument. Specifically, as in Subsection \ref{perzzsec}, we set $N_\varepsilon:=\lfloor L/\varepsilon\rfloor$, and define the discrete time process $\eta^{*\varepsilon,L}$ by \eqref{ggg}, where the underlying  Markov parameters are chosen as in \eqref{Peps}. Comparing \eqref{altchar} and \eqref{altchar2}, it is straightforward to argue from \eqref{etaconv} that
\[\left(\eta^{*\varepsilon,L}_{\lfloor t/\varepsilon \rfloor}\right)_{t\in\mathbb{R}}\buildrel{d}\over{\rightarrow} \left(\eta^{*L}_t\right)_{t\in\mathbb{R}}.\]
The convergence of associated path encodings follows, and the remainder of the proof is identical to Proposition \ref{uuu}.
\end{proof}

We conclude the section be describing the application of the previous result to the ultra-discrete periodic Toda lattice, see \cite{IKT, InTa, KiTo} for background. For this model, we describe the current state by a vector of the form $((Q_j)_{j=1}^J,({E}_j)_{j=1}^{J})  \in (0,\infty)^{2J}$ for some $J\in\mathbb{N}$. Although it appears we have an extra variable to the non-periodic case, this is not so, because we assume that $\sum_{j=1}^J Q_j+\sum_{j=1}^J{E}_j=L$ for some fixed $L\in\mathbb{R}$. Moreover, in order to define the dynamics, we further suppose that $\sum_{j=1}^J Q_j<L/2$, which can be seen as the equivalent condition to requiring fewer than $N/2$ particles in the $N$-periodic BBS model. Introducing the additional notation $({D}_j)_{j=1}^J$ for convenience, the dynamics of the system are given by the following adaptation of \eqref{udtodaeq}:
\begin{align}
(\mathcal{T}Q)_j &:=\min \left\{ Q_j-{D}_j,{E}_j \right\},\label{udtodaper}\\
(\mathcal{T}{E})_j &:=Q_{j+1}+{E}_j-(\mathcal{T}Q)_j,\nonumber,\\
{D}_j&:=\min_{0\leq k\leq J-1}\sum_{l=1}^k\left({E}_{j-l}-Q_{j-l}\right).\nonumber
\end{align}
(In these definitions $((Q_j)_{j=1}^J,({E}_j)_{j=1}^{J})$ are extended periodically to $((Q_j)_{j\in\mathbb{Z}},({E}_j)_{j\in\mathbb{Z}})$.) Given a state vector $((Q_j)_{j=1}^J,({E}_j)_{j=1}^{J})$, we define an associated path encoding $S$ by appealing to the definition at \eqref{abpath} for $t\in[0,L]$, and then concatenating copies of $(S_t)_{t\in[0,L]}$ in such a way that the resulting path is an element of $C(\mathbb{R},\mathbb{R})$. Using this path encoding, the dynamics at \eqref{udtodaper} can be expressed in terms of the operator $\mathcal{T}$ defined as at \eqref{tstarpath}.

\begin{prop}[See {\cite[Theorem 2.3]{CST}}] Fix $J\in\mathbb{N}$ and $L\in(0,\infty)$. Let $((Q_j)_{j=1}^J,({E}_j)_{j=1}^{J})  \in (0,\infty)^{2J}$ satisfy $\sum_{j=1}^J Q_j+\sum_{j=1}^J{E}_j=L$ and $\sum_{j=1}^J Q_j<L/2$, and $S$ be the associated path encoding. It then holds that the periodically transformed configuration $(((\mathcal{T}Q)_j)_{j=1}^J,((\mathcal{T}{E})_j)_{j=1}^{J})$, defined as at \eqref{udtodaper}, has path encoding given by $\mathcal{T}S$, defined as at \eqref{tstarpath}.
\end{prop}

This picture of the ultra-discrete periodic Toda lattice dynamics allows us to deduce the following corollary of Proposition \ref{perpropsl}.

\begin{cor} Fix $J\in\mathbb{N}$, and $A,L\in(0,\infty)$ such that $0<A<L/2$. Let $(\Delta^Q_j)_{j=1}^J$ and $(\Delta^{E}_j)_{j=1}^J$ be independent $\mathrm{Dirichlet}(1,1,\dots,1)$ random variables, and set
\[Q_j:=A\Delta^Q_j,\qquad {E}_j:=(L-A)\Delta^{E}_j,\qquad j=1,\dots,J.\]
It is then the case that $((Q_j)_{j=1}^J,({E}_j)_{j=1}^J)$ is invariant under the dynamics of the ultra-discrete periodic Toda lattice.
\end{cor}
\begin{proof} Fix $\lambda_0,\lambda_1>0$. Let $S^{*L,J}$ be a random path with law equal to that of $S^{*L}$ conditioned on
\begin{equation}\label{tLM}
\#\left\{t\in[0,L):\:t\in LM(S^{*L})\right\}=J,
\end{equation}
and write $Q_1,{E}_1,\dots,Q_J,{E}_J$ for the lengths of the sub-intervals of $[0,L]$ upon which $S^{*L,J}$ has gradient $-1,1,\dots,-1,1$, respectively. Since the left-hand side of \eqref{tLM} is preserved by $\mathcal{T}$ (see \cite[Theorem 2.3]{CST}, for example), it readily follows from Proposition \ref{perpropsl} that $\mathcal{T}S^{*L,J}\buildrel{d}\over{=}S^{*L,J}$, and hence the law of $((Q_j)_{j=1}^J,({E}_j)_{j=1}^J)$ is invariant for the dynamics of the ultra-discrete Toda lattice.

We next aim to identify the distribution of $((Q_j)_{j=1}^J,({E}_j)_{j=1}^J)$ as described in the previous paragraph. By considering the behaviour of the underlying two-sided stationary Markov configuration $\eta$ (that jumps from $i$ to $1-i$ with rate $\lambda_i$, $i=0,1$), it is straightforward to deduce that
\begin{eqnarray}
f_{Q,E}\left((q_j)_{j=1}^J,(e_j)_{j=1}^{J-1}\right)&\propto &\left(\prod_{j=1}^{J}\lambda_1e^{-\lambda_1 q_j}\right)\left(\prod_{j=1}^{J-1}\lambda_0e^{-\lambda_0 e_j}\right)e^{-\lambda_0\left(L-\sum_{j=1}^Jq_j-\sum_{j=1}^{J-1}e_j\right)}\nonumber\\
&\propto&e^{-(\lambda_1-\lambda_0)\sum_{j=1}^Jq_j}\label{abdens}
\end{eqnarray}
for vectors $(q_j)_{j=1}^J$ and $(e_j)_{j=1}^{J-1}$ satisfying $\sum_{j=1}^Jq_j+\sum_{j=1}^{J-1}e_j<L$, $\sum_{j=1}^Jq_j<L/2$ (and the density is zero otherwise). The form of this density suggests the merit of introducing transformed random variables:
\begin{align*}
A&:=\sum_{j=1}^J Q_j,\\
\Delta^Q_j&:=\frac{Q_j}{A},\qquad j=1,\dots,J-1,\\
\Delta^{E}_j&:=\frac{{E}_j}{L-A},\qquad j=1,\dots,J-1.
\end{align*}
Writing $f_{Q,E}$ for the density of the random variables $((Q_j)_{j=1}^J,({E}_j)_{j=1}^{J-1})$, and $f_{\Delta^{Q},A,\Delta^{{E}}}$ for the density of the random variables $((\Delta^Q_j)_{j=1}^{J-1},A,(\Delta^{E}_j)_{j=1}^{J-1})$, we have from a standard change of variable formula:
\[f_{\Delta^{Q},A,\Delta^{{E}}}\left((\delta^Q_j)_{j=1}^{J-1},a,(\delta^{E}_j)_{j=1}^{J-1}\right)=f_{{Q},{E}}\left((q_j)_{j=1}^J,(e_j)_{j=1}^{J-1}\right)\mathrm{Jac}\left((\delta^Q_j)_{j=1}^{J-1},a,(\delta^{E}_j)_{j=1}^{J-1}\right),\]
where
\begin{align*}
(q_1,\dots,q_J)&=\left(a\delta^Q_1,\dots,a\delta^Q_{J-1},a\left(1-\sum_{j=1}^{J-1}\delta^Q_j\right)\right),\\
(e_1,\dots,e_{J-1})&=\left((L-a)\delta^{E}_1,\dots,(L-a)\delta^{E}_{J-1}\right),
\end{align*}
and $\mathrm{Jac}((\delta^Q_j)_{j=1}^{J-1},a,(\delta^{E}_j)_{j=1}^{J-1})$, the Jacobian of the relevant transformation, is given by the modulus of the determinant of the following matrix (all other entries are zero):
\[
\left(
  \begin{array}{ccccccccc}
    a  &   &       &   &            \delta^Q_1        &  & &&\\
       & a &       &   &            \delta^Q_2        &  & &&\\
       &   &\ddots &   &            \vdots                 &  & &&\\
       &   &       & a &            \delta^Q_{J-1}    &  & &&\\
    -a &-a & \dots &-a & 1-\sum_{j=1}^{J-1}\delta^Q_j &  & &&\\
       &   &       &   &            \delta^{E}_1         &  L-a &  &&\\
       &   &       &   &            \delta^{E}_2         &    & L-a  &&\\
       &   &       &   &            \vdots                 &    &    & \ddots &  \\
       &   &       &   &            \delta^{E}_{J-1}     &    &    &        & L-a
  \end{array}
\right).\]
Now, it is elementary to compute this Jacobian to be equal to $(a(L-a))^{(J-1)}$, and thus we obtain from \eqref{abdens} that
\[f_{\Delta^{Q},A,\Delta^{{E}}}\left((\delta^Q_j)_{j=1}^{J-1},a,(\delta^{E}_j)_{j=1}^{J-1}\right)=C(a(L-a))^{(J-1)}e^{-(\lambda_1-\lambda_0)a}\]
for $a<L/2$, $\sum_{j=1}^{J-1}\delta^Q_j<1$, $\sum_{j=1}^{J-1}\delta^{E}_j<1$ (and the density is zero otherwise). Setting $\Delta_j^\kappa=1-\sum_{j=1}^{J-1}\Delta^\kappa_j$ for $\kappa={Q},{E}$, the above formula implies that $A$, $(\Delta^Q_j)_{j=1}^J$ and $(\Delta^{E}_j)_{j=1}^J$ are independent, with $A$ having density
\[f_A(a)=C(a(L-a))^{(J-1)}e^{-(\lambda_1-\lambda_0)a},\qquad a\in(0,L/2),\]
and $(\Delta^Q_j)_{j=1}^J$ and $(\Delta^{E}_j)_{j=1}^J$ being distributed as $\mathrm{Dirichlet}(1,1,\dots,1)$ random variables.

To complete the proof, it remains to condition on the value of $A$. To this end, we first note that, since $A$ is preserved by the dynamics (see \cite[Section 4]{TComm} or \cite[Corollary 2.4]{CST}, for example), we readily obtain that, for any continuous bounded function $F$:
\[\mathbf{E}\left(F\left((((\mathcal{T}Q)_j)_{j=1}^J,((\mathcal{T}{E})_j)_{j=1}^J)\right)\:\vline\:A\right)=\mathbf{E}\left(F\left(((Q_j)_{j=1}^J,({E}_j)_{j=1}^J)\right)\:\vline\:A\right).\]
Moreover, it is straightforward to check that both sides here are continuous in the value of $A$, which means we can interpret the above equation as holding for any fixed, deterministic value of $A\in(0,L/2)$. The desired result follows.
\end{proof}

Finally, we note that a similar conclusion can be drawn for the ultra-discrete periodic Toda lattice whose states are restricted to integer values. Since its proof is almost identical (but slightly easier) to that of the previous corollary, we simply state the result.

\begin{cor} Fix $J\in\mathbb{N}$, and $A,L\in\mathbb{N}$ be such that $J\leq \min\{A,L-A\}$ and also $A<L/2$. Let $(Q_j-1)_{j=1}^J$ and $({E}_j-1)_{j=1}^J$ be independent multinomial random variables with parameters given by $(A-J; J^{-1},J^{-1},\dots,J^{-1})$ and $(L-A-J; J^{-1},J^{-1},\dots,J^{-1})$, respectively. It is then the case that $((Q_j)_{j=1}^J,({E}_j)_{j=1}^J)$ is invariant under the dynamics of the ultra-discrete periodic Toda lattice.
\end{cor}

\section{Conditioned one-sided processes}\label{condsec}

The introduction of Pitman's transformation in \cite{Pitman} was important as it provided a (simple) sample path construction of a three-dimensional Bessel process from a one-dimensional Brownian motion, where the former process can be viewed as Brownian motion conditioned to stay non-negative. Moreover, in the argument of \cite{Pitman}, a discrete analogue of a three-dimensional Bessel process is constructed, and the relation between such a process and random walk conditioned to stay non-negative is explored in detail in \cite{BD}. In this section, we present a general statement that highlights how a statement of invariance under Pitman's transformation for a two-sided process naturally yields an alternative characterisation of the one-sided process conditioned to stay non-negative. The result is particularly transparent in the case of random walks with i.i.d.\ or Markov increments, as well as the zigzag process (details of these examples are presented below). Whilst the applications are not new (cf.\ \cite{HMOC} in particular), we believe it is still worthwhile to present a simple proof of this unified result.

\begin{prop} Let $S=(S_t)_{t\in\mathbb{R}}$ be a random element of $C(\mathbb{R},\mathbb{R})$ that is almost-surely asymptotically linear with strictly positive drift (cf. $\mathcal{S}^{lin}$, as defined at \ref{slin}), and which satisfies $TS\buildrel{d}\over{=}S$. It then holds that
\begin{equation}\label{fffff}
(S_t)_{t\geq 0}\:\vline\:\left\{\inf_{t\geq 0}S_t=0\right\}\buildrel{d}\over{=}(2\bar{M}_t-S_t)_{t\geq 0}\:\vline\:\left\{M_0=0\right\},
\end{equation}
where $\bar{M}=(\bar{M}_t)_{t\geq0}$ is defined by $\bar{M}_t:=\sup_{0\leq s\leq t}S_s$.
\end{prop}
\begin{proof} We have that
\begin{eqnarray*}
(S_t)_{t\geq 0}\:\vline\:\left\{\inf_{t\geq 0}S_t=0\right\}&\buildrel{d}\over{=}&(TS_t)_{t\geq 0}\:\vline\:\left\{\inf_{t\geq 0}TS_t=0\right\}\\
&\buildrel{d}\over{=}&(TS_t)_{t\geq 0}\:\vline\:\left\{M_0=0\right\}\\
&\buildrel{d}\over{=}&(2M_t-S_t)_{t\geq 0}\:\vline\:\left\{M_0=0\right\}\\
&\buildrel{d}\over{=}&(2\bar{M}_t-S_t)_{t\geq 0}\:\vline\:\left\{M_0=0\right\},
\end{eqnarray*}
where the first equality is a consequence of the assumption $TS\buildrel{d}\over{=}S$; the second follows because $\inf_{t\geq 0}TS_t=M_0$ for asymptotically linear $S$ (see \cite[Theorem 2.14]{CKSS}); the third by the definition of $T$ (and the conditioning on $M_0=0$); and the fourth from the observation that $M_t=\max\{\bar{M}_t,M_0\}$ for $t\geq 0$.
\end{proof}

\begin{rem} The condition of asymptotic linearity is sufficient but not necessary for the above proof to work. The relation between the future infimum of $TS$ and the past maximum of $S$ holds whenever $S$ is in the domain of $T$ and $T^{-1}TS=S$. See \cite[Theorem 2.14]{CKSS} for details.
\end{rem}

\begin{rem} The same result holds for paths $S:\mathbb{Z}\rightarrow\mathbb{R}$ whose increments take values either $-1$ or $1$. For more general increments, the argument does not apply (since the future infimum of $TS$ and the past maximum of $S$ do not necessarily agree).
\end{rem}

\begin{exmp} The simplest non-trivial application of the previous result (and the previous remark) is when $S$ is a simple random walk with i.i.d.\ Bernoulli increments and strictly positive drift (i.e.\ the path encoding of Subsection \ref{iidsec}). In this case, the right-hand side of \eqref{fffff} can be replaced by the unconditioned process.
\end{exmp}

\begin{exmp} We next consider the case when $S$ is a path with Markovian increments of the form described in Subsection \ref{markovsec}. In this case, the conditioning on right-hand side of \eqref{fffff} can be replaced by the initial condition $\eta_0=0$ (using the BBS notation of earlier sections).
\end{exmp}

\begin{exmp} For $S$ the zigzag process of Subsection \ref{zzsec}, the result applies, and the conditioning on right-hand side of \eqref{fffff} can also be replaced by the initial condition $\eta_0=0$ (i.e.\ $S$ has a gradient of $1$ at 0).
\end{exmp}

\section*{Acknowledgements}

DC would like to acknowledge the support of his JSPS Grant-in-Aid for Research Activity Start-up, 18H05832, and MS would like to acknowledge the support of her JSPS Grant-in-Aid for Scientific Research (B), 16KT0021.

\providecommand{\bysame}{\leavevmode\hbox to3em{\hrulefill}\thinspace}
\providecommand{\MR}{\relax\ifhmode\unskip\space\fi MR }
\providecommand{\MRhref}[2]{%
  \href{http://www.ams.org/mathscinet-getitem?mr=#1}{#2}
}
\providecommand{\href}[2]{#2}


\begin{thebibliography}{10}

\bibitem{Albenque}
M.~Albenque, \emph{A note on the enumeration of directed animals via gas
  considerations}, Ann. Appl. Probab. \textbf{19} (2009), no.~5, 1860--1879.

\bibitem{BD}
J.~Bertoin and R.~A. Doney, \emph{On conditioning a random walk to stay
  nonnegative}, Ann. Probab. \textbf{22} (1994), no.~4, 2152--2167.

\bibitem{Burke}
P.~J. Burke, \emph{The output of a queuing system}, Operations Res. \textbf{4}
  (1956), 699--704 (1957).

\bibitem{CKSS}
D.~A. Croydon, T.~Kato, M.~Sasada, and S.~Tsujimoto, \emph{Dynamics of the
  box-ball system with random initial conditions via pitman's transformation},
  preprint appears at arXiv:1806.02147, 2018.

\bibitem{CST}
D.~A. Croydon, M.~Sasada, and S.~Tsujimoto, \emph{Dynamics of the
  ultra-discrete {T}oda lattice via {P}itman's transformation}, forthcoming.
%Update arXiv reference later.

\bibitem{DZ}
A.~Dembo and O.~Zeitouni, \emph{Large deviations techniques and applications},
  Stochastic Modelling and Applied Probability, vol.~38, Springer-Verlag,
  Berlin, 2010, Corrected reprint of the second (1998) edition.

\bibitem{DMOC}
M.~Draief, J.~Mairesse, and N.~O'Connell, \emph{Queues, stores, and tableaux},
  J. Appl. Probab. \textbf{42} (2005), no.~4, 1145--1167.

\bibitem{Ferrari}
P.~A. Ferrari, C.Nguyen, L.~Rolla, and M.~Wang, \emph{Soliton decomposition of
  the box-ball system}, preprint appears at arXiv:1806.02798, 2018.

\bibitem{FG}
P.~A. Ferrari and D.~Gabrielli, \emph{Bbs invariant measures with independent
  soliton components}, preprint appears at arXiv:1812.02437, 2018.

\bibitem{GT}
P.~W. Glynn and H.~Thorisson, \emph{Two-sided taboo limits for {M}arkov
  processes and associated perfect simulation}, Stochastic Process. Appl.
  \textbf{91} (2001), no.~1, 1--20.

\bibitem{GT2}
\bysame, \emph{Structural characterization of taboo-stationarity for general
  processes in two-sided time}, Stochastic Process. Appl. \textbf{102} (2002),
  no.~2, 311--318.

\bibitem{HMOC}
B.~M. Hambly, J.~B. Martin, and N.~O'Connell, \emph{Pitman's {$2M-X$} theorem
  for skip-free random walks with {M}arkovian increments}, Electron. Comm.
  Probab. \textbf{6} (2001), 73--77.

\bibitem{Har}
T.~E. Harris, \emph{Random measures and motions of point processes}, Z.
  Wahrscheinlichkeitstheorie und Verw. Gebiete \textbf{18} (1971), 85--115.

\bibitem{HW}
J.~M. Harrison and R.~J. Williams, \emph{On the quasireversibility of a
  multiclass {B}rownian service station}, Ann. Probab. \textbf{18} (1990),
  no.~3, 1249--1268.

\bibitem{IKT}
R.~Inoue, A.~Kuniba, and T.~Takagi, \emph{Integrable structure of box-ball
  systems: crystal, {B}ethe ansatz, ultradiscretization and tropical geometry},
  J. Phys. A \textbf{45} (2012), no.~7, 073001, 64.

\bibitem{InTa}
R.~Inoue and T.~Takenawa, \emph{Tropical spectral curves and integrable
  cellular automata}, Int. Math. Res. Not. IMRN (2008), no.~9, Art ID. rnn019,
  27.

\bibitem{Kac}
M.~Kac, \emph{A stochastic model related to the telegrapher's equation}, Rocky
  Mountain J. Math. \textbf{4} (1974), 497--509, Reprinting of an article
  published in 1956, Papers arising from a Conference on Stochastic
  Differential Equations (Univ. Alberta, Edmonton, Alta., 1972).

\bibitem{KiTo}
T.~Kimijima and T.~Tokihiro, \emph{Initial-value problem of the discrete
  periodic {T}oda equation and its ultradiscretization}, Inverse Problems
  \textbf{18} (2002), no.~6, 1705--1732.

\bibitem{Lev}
L.~Levine, H.~Lyu, and J.~Pike, \emph{Double jump phase transition in a soliton
  cellular automaton}, preprint appears at arXiv:1706.05621, 2017.

\bibitem{NTS}
A.~Nagai, T.~Tokihiro, and J.~Satsuma, \emph{Ultra-discrete {T}oda molecule
  equation}, Phys. Lett. A \textbf{244} (1998), no.~5, 383--388.

\bibitem{OCY}
N.~O'Connell and M.~Yor, \emph{Brownian analogues of {B}urke's theorem},
  Stochastic Process. Appl. \textbf{96} (2001), no.~2, 285--304.

\bibitem{Pitman}
J.~W. Pitman, \emph{One-dimensional {B}rownian motion and the three-dimensional
  {B}essel process}, Advances in Appl. Probability \textbf{7} (1975), no.~3,
  511--526.

\bibitem{PS}
S.~C. Port and C.~J. Stone, \emph{Infinite particle systems}, Trans. Amer.
  Math. Soc. \textbf{178} (1973), 307--340.

\bibitem{RDYO}
M.~Rigol, V.~Dunjko, V.~Yurovsky, and M.~Olshanii, \emph{Relaxation in a
  completely integrable many-body quantum system: An ab initio study of the
  dynamics of the highly excited states of 1d lattice hard-core bosons}, Phys.
  Rev. Lett. \textbf{98} (2007), 050405.

\bibitem{RMO}
M.~Rigol, A.~Muramatsu, and M.~Olshanii, \emph{Hard-core bosons on optical
  superlattices: Dynamics and relaxation in the superfluid and insulating
  regimes}, Phys. Rev. A \textbf{74} (2006), 053616.

\bibitem{TComm}
T.~Takagi, \emph{Commuting time evolutions in the tropical periodic toda
  lattice}, Journal of the Physical Society of Japan \textbf{81} (2012),
  no.~10, 104005.

\bibitem{takahashi1990}
D.~Takahashi and J.~Satsuma, \emph{A soliton cellular automaton}, J. Phys. Soc.
  Japan \textbf{59} (1990), 3514--3519.

\bibitem{T}
T.~Tokihiro, \emph{Ultradiscrete systems (cellular automata)}, Discrete
  integrable systems, Lecture Notes in Phys., vol. 644, Springer, Berlin, 2004,
  pp.~383--424.

\bibitem{TT}
\bysame, \emph{The mathematics of box-ball systems}, Asakura Shoten, 2010.

\bibitem{TTMS}
T.~Tokihiro, D.~Takahashi, J.~Matsukidaira, and J.~Satsuma, \emph{From soliton
  equations to integrable cellular automata through a limiting procedure},
  Phys. Rev. Lett. \textbf{76} (1996), no.~18, 3247--3250.

\bibitem{Torii}
M.~Torii, D.~Takahashi, and J.~Satsuma, \emph{Combinatorial representation of
  invariants of a soliton cellular automaton}, Phys. D \textbf{92} (1996),
  no.~3, 209 -- 220.

\bibitem{VR}
L.~Vidmar and M.~Rigol, \emph{Generalized gibbs ensemble in integrable lattice
  models}, Journal of Statistical Mechanics: Theory and Experiment
  \textbf{2016} (2016), no.~6, 064007.

\bibitem{YYT}
D.~Yoshihara, F.~Yura, and T.~Tokihiro, \emph{Fundamental cycle of a periodic
  box-ball system}, J. Phys. A \textbf{36} (2003), no.~1, 99--121.

\bibitem{YT}
F.~Yura and T.~Tokihiro, \emph{On a periodic soliton cellular automaton}, J.
  Phys. A \textbf{35} (2002), no.~16, 3787--3801.

\end{thebibliography}
\end{document}